\documentclass[a4paper,twoside]{article}
\usepackage{a4}
\usepackage{amssymb}
\usepackage{amsmath}
\usepackage{comment}
\usepackage{upref}
\usepackage[dvips,colorlinks,citecolor=blue,linkcolor=blue]{hyperref}
\usepackage[dvipsnames]{color}
\usepackage[active]{srcltx}
\allowdisplaybreaks[2] 
\newcount\minutes \newcount\hours
\hours=\time \divide\hours 60 \minutes=\hours
 \multiply\minutes-60
\advance\minutes \time
\newcommand{\klockan}{\the\hours:{\ifnum\minutes<10 0\fi}\the\minutes}
\newcommand{\tid}{\today\ \klockan}
\newcommand{\prtid}{\smash{\raise 10mm \hbox{\LaTeX ed \tid}}}
\renewcommand{\prtid}{}
\makeatletter  \headheight 10pt
\def\sectionmark#1{} 
\def\subsectionmark#1{}
\newcommand{\sectnr}{\ifnum \c@secnumdepth >\z@
                 \thesection.\hskip 1em\relax \fi}
\def\@evenhead{\footnotesize\rm\thepage\hfil\leftmark\hfil\llap{\prtid}}
\def\@oddhead{\footnotesize\rm\rlap{\prtid}\hfil\rightmark\hfil\thepage}
\def\tableofcontents{\section*{Contents} 
 \@starttoc{toc}}
\makeatother
\makeatletter
\def\@biblabel#1{#1.}
\makeatother
\makeatletter
\let\Thebibliography=\thebibliography
\renewcommand{\thebibliography}[1]{\def\@mkboth##1##2{}\Thebibliography{#1}
\addcontentsline{toc}{section}{References}
\frenchspacing 
\setlength{\@topsep}{0pt}
\setlength{\itemsep}{0pt}%
\setlength{\parskip}{0pt plus 2pt}%
} \makeatother
\makeatletter
\def\mdots@{\mathinner.\nonscript\!.%
 \ifx\next,.\else\ifx\next;.\else\ifx\next..\else
 \nonscript\!\mathinner.\fi\fi\fi}
\let\ldots\mdots@
\makeatother
\makeatletter
\let\Enumerate=\enumerate
\renewcommand{\enumerate}{\Enumerate%
\setlength{\@topsep}{0pt}
\setlength{\itemsep}{0pt}%
\setlength{\parskip}{0pt plus 1pt}%
\renewcommand{\theenumi}{\textup{(\alph{enumi})}}%
\renewcommand{\labelenumi}{\theenumi}%
}
\let\endEnumerate=\endenumerate
\renewcommand{\endenumerate}{\endEnumerate\unskip}
\makeatother
\makeatletter
\def\@seccntformat#1{\csname the#1\endcsname.\quad}
\makeatother
\newcommand{\authortitle}[2]{\author{#1}\title{#2}\markboth{#1}{#2}}
\newcommand{\art}[6]{{\sc #1, \rm #2, \it #3 \bf #4 \rm (#5), \mbox{#6}.}}
\newcommand{\auth}[2]{{#1, #2.}}
\newcommand{\idxauth}[2]{{#1, #2.}}
\newcommand{\artin}[3]{{\sc #1, \rm #2,  in #3.}}
\newcommand{\artprep}[3]{{\sc #1, \rm #2, #3.}}
\newcommand{\arttoappear}[3]{{\sc #1, \rm #2, to appear in \it #3}}
\newcommand{\artnopt}[6]{{\sc #1, \rm #2. \it #3\/ \bf #4 \rm (#5), \mbox{#6}}}
\newcommand{\book}[3]{{\sc #1, \it #2, \rm #3.}}
\newcommand{\AND}{{\rm and }}
\def\cprime{{\mathsurround0pt$'$}}
\RequirePackage{amsthm}
\newtheoremstyle{descriptive}%
  {\topsep}   
  {\topsep}   
  {\rmfamily} 
  {}          
  {\bfseries} 
  {.}         
  { }         
  {}          
\newtheoremstyle{propositional}%
  {\topsep}   
  {\topsep}   
  {\itshape}  
  {}          
  {\bfseries} 
  {.}         
  { }         
  {}          
\theoremstyle{propositional}
\newtheorem{thm}{Theorem}[section]
\newtheorem{theorem}[thm]{Theorem}  
\newtheorem{prop}[thm]{Proposition}
\newtheorem{lem}[thm]{Lemma}
\newtheorem{lemma}[thm]{Lemma} 
\theoremstyle{descriptive}
\newtheorem{deff}[thm]{Definition}
\newtheorem{remark}[thm]{Remark}
\makeatletter
\renewenvironment{proof}[1][\proofname]{\par
  \pushQED{\qed}%
  \normalfont
  \trivlist
  \item[\hskip\labelsep
        \itshape
    #1\@addpunct{.}]\ignorespaces
}{%
  \popQED\endtrivlist\@endpefalse
} \makeatother
\newcommand{\setm}{\setminus}
\renewcommand{\emptyset}{\varnothing}
\def\vint{\mathop{\mathchoice%
          {\setbox0\hbox{$\displaystyle\intop$}\kern 0.22\wd0%
           \vcenter{\hrule width 0.6\wd0}\kern -0.82\wd0}%
          {\setbox0\hbox{$\textstyle\intop$}\kern 0.2\wd0%
           \vcenter{\hrule width 0.6\wd0}\kern -0.8\wd0}%
          {\setbox0\hbox{$\scriptstyle\intop$}\kern 0.2\wd0%
           \vcenter{\hrule width 0.6\wd0}\kern -0.8\wd0}%
          {\setbox0\hbox{$\scriptscriptstyle\intop$}\kern 0.2\wd0%
           \vcenter{\hrule width 0.6\wd0}\kern -0.8\wd0}}%
          \mathopen{}\int}
\newcommand{\Cp}{{C_p}}
\DeclareMathOperator{\capp}{cap}
\newcommand{\cp}{\capp_p}
\DeclareMathOperator{\diam}{diam} 
\DeclareMathOperator{\Lip}{Lip}

\DeclareMathOperator*{\essinf}{ess\,inf}
\DeclareMathOperator*{\esssup}{ess\,sup}
\DeclareMathOperator*{\finelim}{fine\,lim}
\DeclareMathOperator*{\essliminf}{ess\,lim\,inf}
\DeclareMathOperator{\fineint}{fine-int}

\newcommand{\bdry}{\partial}
\newcommand{\bdy}{\bdry}
\newcommand{\loc}{_{\rm loc}}
{\catcode`p =12 \catcode`t =12 \gdef\eeaa#1pt{#1}}      
\def\accentadjtext#1{\setbox0\hbox{$#1$}\kern   
                \expandafter\eeaa\the\fontdimen1\textfont1 \ht0 }
\def\accentadjscript#1{\setbox0\hbox{$#1$}\kern 
                \expandafter\eeaa\the\fontdimen1\scriptfont1 \ht0 }
\def\accentadjscriptscript#1{\setbox0\hbox{$#1$}\kern   
                \expandafter\eeaa\the\fontdimen1\scriptscriptfont1 \ht0 }
\def\accentadjtextback#1{\setbox0\hbox{$#1$}\kern       
                -\expandafter\eeaa\the\fontdimen1\textfont1 \ht0 }
\def\accentadjscriptback#1{\setbox0\hbox{$#1$}\kern     
                -\expandafter\eeaa\the\fontdimen1\scriptfont1 \ht0 }
\def\accentadjscriptscriptback#1{\setbox0\hbox{$#1$}\kern 
                -\expandafter\eeaa\the\fontdimen1\scriptscriptfont1 \ht0 }
\def\itoverline#1{{\mathsurround0pt\mathchoice
        {\rlap{$\accentadjtext{\displaystyle #1}
                \accentadjtext{\vrule height1.593pt}
                \overline{\phantom{\displaystyle #1}
                \accentadjtextback{\displaystyle #1}}$}{#1}}
        {\rlap{$\accentadjtext{\textstyle #1}
                \accentadjtext{\vrule height1.593pt}
                \overline{\phantom{\textstyle #1}
                \accentadjtextback{\textstyle #1}}$}{#1}}
        {\rlap{$\accentadjscript{\scriptstyle #1}
                \accentadjscript{\vrule height1.593pt}
                \overline{\phantom{\scriptstyle #1}
                \accentadjscriptback{\scriptstyle #1}}$}{#1}}
        {\rlap{$\accentadjscriptscript{\scriptscriptstyle #1}
                \accentadjscriptscript{\vrule height1.593pt}
                \overline{\phantom{\scriptscriptstyle #1}
                \accentadjscriptscriptback{\scriptscriptstyle #1}}$}{#1}}}}
\newcommand{\dmu}{d\mu}
\newcommand{\de}{\delta}
\newcommand{\eps}{\varepsilon}
\newcommand{\la}{\lambda}
\newcommand{\ga}{\gamma}
\newcommand{\s}{\sigma}
\newcommand{\Om}{\Omega}
\renewcommand{\phi}{\varphi}
\newcommand{\p}{{$p\mspace{1mu}$}}
\newcommand{\clEp}{{\itoverline{E}\mspace{1mu}}^p}
\newcommand{\clGjp}{\itoverline{G}_j^{\mspace{2mu}p}}
\newcommand{\clGkp}{\itoverline{G}_k^{\mspace{2mu}p}}
\newcommand{\R}{\mathbf{R}}
\newcommand{\eR}{{\overline{\R}}}
\newcommand{\T}{{\cal T}}
\newcommand{\K}{{\cal K}}
\newcommand{\oHpind}[1]{H_{#1}}     
\newcommand{\limplus}{{\mathchoice{\raise.17ex\hbox{$\scriptstyle +$}}
		{\raise.17ex\hbox{$\scriptstyle +$}}
		{\raise.1ex\hbox{$\scriptscriptstyle +$}}
		{\scriptscriptstyle +}}}
\newcommand{\Np}{N^{1,p}}
\newcommand{\Nploc}{N^{1,p}\loc}
\newcommand{\Ct}{\widetilde{C}}
\makeatletter
\newcommand{\setcurrentlabel}[1]{\def\@currentlabel{#1}}
\makeatother
\numberwithin{equation}{section}
\newcommand{\Ga}{\Gamma}
\newcommand{\Lploc}{L^p\loc}

\newcommand{\eqv}{\ensuremath{
\mathchoice{\quad \Longleftrightarrow \quad}{\Leftrightarrow}
                 {\Leftrightarrow}{\Leftrightarrow}} }
\newcommand{\imp}{\ensuremath{\Rightarrow} }

\newenvironment{ack}{\medskip{\it Acknowledgement.}}{}

\begin{document}

\authortitle{Anders Bj\"orn, Jana Bj\"orn and Visa Latvala}
{The weak Cartan property for the \p-fine topology on metric spaces}
\author{
Anders Bj\"orn \\
\it\small Department of Mathematics, Link\"opings universitet, \\
\it\small SE-581 83 Link\"oping, Sweden\/{\rm ;}
\it \small anders.bjorn@liu.se
\\
\\
Jana Bj\"orn \\
\it\small Department of Mathematics, Link\"opings universitet, \\
\it\small SE-581 83 Link\"oping, Sweden\/{\rm ;}
\it \small jana.bjorn@liu.se
\\
\\
Visa Latvala \\
\it\small Department of Physics and Mathematics,
University of Eastern Finland, \\
\it\small P.O. Box 111, FI-80101 Joensuu,
Finland\/{\rm ;} 
\it \small visa.latvala@uef.fi
}

\date{}

\maketitle

\noindent{\small
 {\bf Abstract}.
We study the \p-fine topology on
complete metric spaces  equipped with a doubling measure
supporting a \p-Poincar\'e inequality, $1 < p< \infty$.
We establish a weak Cartan property, which
yields characterizations of the \p-thinness and the \p-fine continuity, and
allows us to show that the \p-fine topology is the coarsest topology making
all \p-superharmonic functions continuous.
Our \p-harmonic and superharmonic functions are defined by means of
scalar-valued upper gradients and do not rely on a vector-valued
differentiable structure.}

\bigskip


\medskip


\section{Introduction}

The aim of this paper is to
study the \p-fine topology and the fine potential theory
associated with \p-harmonic functions on a
complete metric space $X$ equipped with a doubling measure $\mu$
supporting a \p-Poincar\'e inequality, $1 < p< \infty$.

Nonlinear potential theory associated with \p-harmonic functions
has been studied since the 1960s.
For extensive treatises and notes on the history,
see the monographs Adams--Hedberg~\cite{AdHe} and
Heinonen--Kilpel\"ainen--Martio~\cite{HeKiMa}, the latter developing
the theory on weighted $\R^n$
(with respect to \p-admissible weights).
Starting in the 1990s a lot of attention
has been given to analysis on metric spaces,
see e.g.\ Haj\l asz~\cite{Haj-PA}, \cite{Haj03},
Haj\l asz--Koskela~\cite{HaKo}, Heinonen~\cite{heinonen}, \cite{hei-BAMS},
and Heinonen--Koskela~\cite{HeKo98}.
Around 2000 this initiated 
studies of \p-harmonic and \p-superharmonic functions on
metric spaces without a differentiable structure, by e.g.\ Shanmugalingam~\cite{Sh-harm},
Kinnunen--Martio~\cite{KiMa02}, Kinnunen--Shanmugalingam~\cite{KiSh01},
Bj\"orn--MacManus--Shanmugalingam~\cite{BMS}
and  Bj\"orn--Bj\"orn--Shanmugalingam~\cite{BBS}, \cite{BBS2}.
The theory has later been further developed by these and other authors,
see the monograph Bj\"orn--Bj\"orn~\cite{BBbook} and the references therein.

While \p-harmonic functions are known to be locally H\"older continuous
(even on metric spaces, see~\cite{KiSh01}),
\p-superharmonic functions are in general only lower semicontinuous.
However, at points of  discontinuity they still exhibit more
regularity than just lower semicontinuity, namely, the limit
$\lim u(x)$, as $x\to x_0$, exists along a substantial (in a capacitary
sense) part of $x_0$'s neighbourhood and equals $u(x_0)$.
The topology giving rise to such neighbourhoods and
limits is called the \p-\emph{fine topology}.
Together with the associated fine \emph{potential theory}
it goes back to Cartan in the 1940s
in the linear case $p=2$, which has been later systematically studied, 
see e.g.\ Fuglede~\cite{Fugl71}, \cite{Fug} and 
Luke\v{s}--Mal\'y--Zaj\'i\v{c}ek~\cite{LuMaZa}.

The nonlinear fine potential theory started in 
the 1970s, with papers by e.g.\ Maz\cprime ya~\cite{Maz70}, 
Maz\cprime ya--Havin~\cite{MazHa70}, \cite{MazHa72}, 
Hedberg~\cite{Hedb}, Adams--Meyers~\cite{AdMey}, Meyers~\cite{Mey75},
Hedberg--Wolff~\cite{HedWol}, Adams--Lewis~\cite{AdLew}
and Lindqvist--Martio~\cite{Lind-Mar}.
See also the notes to Chapter~12 in
Heinonen--Kilpel\"ainen--Martio~\cite{HeKiMa} and Section~2.6 in
Mal\'y--Ziemer~\cite{MZ}.
In the 1990s the fine potential theory associated
with \p-harmonic functions
was developed further in
Heinonen--Kilpel\"ainen--Martio~\cite{HeKiMa89},
Kilpel\"ainen--Mal\'y~\cite{KiMa92}, \cite{KiMa},
Latvala~\cite{LatPhD}, \cite{Lat97}, \cite{Lat00},
and the monograph Mal\'y--Ziemer~\cite{MZ} for unweighted $\R^n$.
The monograph \cite{HeKiMa}
is the main source for fine potential theory on weighted $\R^n$
(note that Chapter~21, which is only in the second addition, contains
some more recent results). See
also Mikkonen~\cite{Mikkonen} for related results (in weighted $\R^n$)
on the Wolff
potential.
In fact, the Wolff potential appeared already
in  Maz\cprime ya--Havin~\cite{MazHa72}.

The fine potential theory in metric spaces is more recent, starting with
Kinnunen--Latvala~\cite{KiLa}, J.~Bj\"orn~\cite{JB-pfine} and
Korte~\cite{korte08}, where it was shown that \p-superharmonic functions on
open subsets of metric spaces are \p-finely continuous.
There are also some related more recent results
in Bj\"orn--Bj\"orn~\cite{BBnonopen} and~\cite{BBvarcap}.
As in the classical situation, the \p-fine topology on metric spaces is
defined by means of \p-capacity and \p-thin sets, see Section~\ref{sect-fine-cont}.

From now on we drop the $p$ from the notation and just write
e.g.\ fine and superharmonic even though the notions depend
on $p$. Our first main result complements the results in
\cite{JB-pfine}, \cite{HeKiMa89}, \cite{KiLa}
and \cite{korte08} as follows.

\begin{thm} \label{thm-coarsest-intro}
The fine topology is the coarsest topology making all
superharmonic functions on open subsets of $X$ continuous.
\end{thm}

The superharmonic functions considered in this  and most of the earlier papers
on metric spaces are
defined through upper gradients (see later sections for
precise definitions), which in particular means
that we have \emph{no equation}, only  variational inequalities,
to work with.
In this way the results do not depend on any differentiable structure
of the metric space.

The proofs of our main results are based on pointwise estimates of
capacitary potentials. These estimates lead in a natural way to a
central property which we call the \emph{weak Cartan property},
see Theorem~\ref{thm-weak-Cartan}.
The following consequence is a slight reformulation and extension
of the weak Cartan property.

\begin{thm} \label{thm-j}
Let $E \subset X$ be an arbitrary set, and
let $x_0 \in \itoverline{E} \setm E$.
Then the following are equivalent\/\textup{:}
\begin{enumerate}
\item \label{j-b}
$E$ is thin at $x_0$\textup{;}
\item \label{j-a}
$x_0 \notin \clEp$, where $\clEp$ is the fine closure of $E$\textup{;}
\item \label{j-aa}
$X \setm E$ is a fine neighbourhood of $x_0$\textup{;}
\item \label{j-k}
there are $k\ge 2$ superharmonic functions $u_1,\ldots, u_k$
in an open neighbourhood
of $x_0$ such that the function $v=\max\{u_1,\ldots,u_k\}$ satisfies
\begin{equation}  \label{eq-Cartan}
     v(x_0) < \liminf_{E \ni x \to x_0} v(x) \textup{;}
\end{equation}
\item \label{j-c}
condition \ref{j-k} holds with $k=2$ nonnegative bounded superharmonic
functions. 
\end{enumerate}
\end{thm}

Here and elsewhere, a set $U$  is a \emph{fine neighbourhood}
of a point $x_0$ if it contains
a  finely open set $V \ni x_0$; it is not
required that $U$ itself is finely open.
Note also that if $x_0\in E$, then $E$ is thin at $x_0$
if and only if $\Cp(\{x_0\})=0$ and $E \setm \{x_0\}$ is thin at $x_0$.
This is a consequence of the following generalization of
Theorem~6.33 in Heinonen--Kilpel\"ainen--Martio~\cite{HeKiMa}.

\begin{prop} \label{prop-pt-thick}
If $\Cp(\{x_0\})>0$, then $\{x_0\}$ is thick at $x_0$.
\end{prop}

Note that the converse statement is trivially true.
At points with positive capacity we further improve Theorem~\ref{thm-j}
and obtain the usual Cartan property (with $k=1$),
see Proposition~\ref{prop-char-thin-cap-pos}.
(Note that in weighted $\R^n$ and in metric spaces it can happen that some
points have positive capacity while others do not.
A sharp condition for when $\Cp(\{x_0\})>0$ is given
in Proposition~8.3 in Bj\"orn--Bj\"orn--Lehrb\"ack~\cite{ringcap}.)
Proposition~\ref{prop-char-thin-cap-pos} also shows
 that $E$ is thin at $x_0 \in \itoverline{E} \setm E$
with $\Cp(\{x_0\})>0$ if and only if the seemingly weaker condition
\[
\lim_{\rho\to0} \Cp(E\cap B(x_0,\rho))=0
\]
holds.
This characterization fails for points with zero capacity.

The classical \emph{Cartan property} says that if $E\subset\R^n$ is thin at
$x_0\in \itoverline{E}\setm E$, then for every $r>0$ there is a nonnegative
bounded superharmonic function $u$ on $B(x_0,r)$ such that
\[
u(x_0)<\liminf_{E\ni x\to x_0} u(x),
\]
see Theorem~1.3 in Kilpel\"ainen--Mal\'y~\cite{KiMa}
or
Theorem~2.130 in Mal\'y--Ziemer~\cite{MZ} for
the nonlinear case on unweighted $\R^n$,
and Theorem~21.26 in Heinonen--Kilpel\"ainen--Martio~\cite{HeKiMa}
(only in the second edition) for weighted $\R^n$.
In the generality of this paper,
for superharmonic functions defined through upper
gradients on metric spaces, 
it is not known whether the classical Cartan property (with $k=1$)
holds, since its proof is based on the equation rather than on the
minimization problem.
Using variational methods, we have only been able to prove it for points
with positive capacity in Proposition~\ref{prop-char-thin-cap-pos}.
However, the weak Cartan property
 provides us with two superharmonic
functions whose maximum in many situations can be
used instead of the usual Cartan property
(but not always, since the maximum
need not be superharmonic).
In particular Theorem~\ref{thm-coarsest-intro} follows quite
easily.

The (strong) Cartan property is closely related to the necessity part of the Wiener criterion, as it provides a superharmonic function which is not continuous at $x_0$, and can thus be used to obtain a   \p-harmonic function which does not attain its continuous boundary values at $x_0$.
The weak Cartan property only leads to the necessity part of the Wiener criterion for certain domains, see Remark~\ref{rem-Wolff}.
Due to the lack of equation, the necessity part of the Wiener criterion for general domains in metric spaces is not known for \p-harmonic functions defined by means of upper gradients, while for Cheeger \p-harmonic functions based on a vector-valued differentiable structure it was proved in J.~Bj\"orn~\cite{JB-Matsue}. 
The sufficiency part of the Wiener criterion in metric spaces was proved in 
Bj\"orn--MacManus--Shanmugalingam~\cite{BMS} and J.~Bj\"orn~\cite{JB-pfine}.
In Euclidean spaces, the Wiener criterion was obtained in  Maz\cprime ya~\cite{Maz70}, Lindqvist--Martio~\cite{Lind-Mar}, Heinonen--Kilpel\"ainen--Martio~\cite{HeKiMa}, Kilpel\"ainen--Mal\'y~\cite{KiMa92} and Mikkonen~\cite{Mikkonen}.

The outline of the paper is as follows:
In Sections~\ref{sect-prelim} and~\ref{sect-superharm}
we introduce the necessary background on metric spaces, upper gradients,
Newtonian spaces, capacity and superharmonic functions.
In Section~\ref{sect-fine-cont} we introduce the fine topology, cite
the necessary background results, and establish a number of auxiliary results
not requiring the weak Cartan property nor the capacitary estimates
used to establish it.
We also conclude the following generalization
of a result by J.~Bj\"orn~\cite{JB-pfine} and Korte~\cite{korte08},
who (independently) established the result corresponding to (b)
for open sets $U$, see Theorem~\ref{thm-qcont-fine}.

\begin{thm}\label{thm-fine-quasicont-fine}
{\rm(a)} Any quasiopen set $U\subset X$ can be written as
\(U=V\cup E\),
where $V$ is finely open and $\Cp(E)=0$.

{\rm(b)} Let $u$ be a quasicontinuous function on a quasiopen or 
finely open set $U$.
Then $u$ is finely continuous q.e.\ in $U$.
\end{thm}

A fundamental step in the proof is the fact that the
capacity of a set coincides with the capacity of its fine closure,
see Lemma~\ref{lemma-capfineclosure} which generalizes Corollary~4.5
in J.~Bj\"orn~\cite{JB-pfine}.

Section~\ref{sect-weak-Cartan} is devoted to
the proof of the weak Cartan property (Theorem~\ref{thm-weak-Cartan}).
Also Theorem~\ref{thm-j} is established.
In the last section, Section~\ref{sect-coarsest},
we draw a number of consequences of the weak Cartan property,
including Theorem~\ref{thm-coarsest-intro}
and Proposition~\ref{prop-pt-thick},
and end the paper by proving the following characterization of fine
continuity, which as pointed out in
Mal\'y--Ziemer~\cite{MZ} is by no means trivial.

\begin{thm}\label{thm-char-fine-cont}
Let $u$ be a function on a fine neighbourhood $U$ of $x_0$.
Then the following conditions are equivalent\/\textup{:}
 \begin{enumerate}
 \item \label{it-i}
 $u$ is finely continuous at $x_0$\/\textup{;}
 \item \label{it-ii}
the set $\{x\in U:|u(x)-u(x_0)|\ge\varepsilon\}$
is thin at $x$ for each $\varepsilon>0$\/\textup{;}
 \item \label{it-iii} there exists a set $E$ which is thin at $x_0$
such that
 $$
    u(x_0)=\lim_{U\setminus E \ni x\to x_0} u(x),
 $$
 where the limit is taken with respect to the metric topology.
 \end{enumerate}
\end{thm}

Many of the results in this paper
are known on weighted $\R^n$, but as
far as we know, Theorem~\ref{thm-fine-quasicont-fine}
and
Proposition~\ref{prop-char-thin-cap-pos} are new on weighted $\R^n$ and Lemma~\ref{lemma-capfineclosure} is new even on unweighted $\R^n$.
Note also that many of our proofs in Sections~\ref{sect-weak-Cartan} and \ref{sect-coarsest} differ from the proofs on weighted $\R^n$, since our approach is purely based on variational inequalities, not on an equation.
The proofs of the auxiliary results in Section~\ref{sect-fine-cont}
are analogous to the Euclidean ones, but we have given proofs whenever some technical modifications are required.

\begin{ack}
The  first two authors were supported by the Swedish Research Council.
Part of this research was done during several visits
of the third author to Link\"opings universitet in 2009, 2012 and 2013.
The first of these visits was supported by the
Scandinavian Research Network \emph{Analysis and Application},
and the others by Link\"opings universitet. 
The paper was completed while all three authors
visited Institut Mittag-Leffler in the autumn of 2013.
They want to thank the institute for the hospitality,
and the third author also wishes to thank the 
Department of Mathematics at
Link\"opings universitet for its hospitality.
\end{ack}

\section{Notation and preliminaries}
\label{sect-prelim}

We assume throughout the paper that $1 < p<\infty$
and that $X=(X,d,\mu)$ is a metric space equipped
with a metric $d$ and a positive complete  Borel  measure $\mu$
such that $0<\mu(B)<\infty$ for all (open) balls $B \subset X$.
The $\sigma$-algebra on which $\mu$ is defined
is obtained by the completion of the Borel $\sigma$-algebra.
It follows that $X$ is separable.

Towards the end of the section we further assume
that $X$ is complete and  supports a \p-Poincar\'e inequality, and that
$\mu$ is doubling,
which are then assumed throughout the rest of
the paper.
We also always assume that $\Om \subset X$ is a nonempty open
set.

We say that $\mu$  is \emph{doubling} if
there exists a \emph{doubling constant} $C>0$ such that for all balls
$B=B(x_0,r):=\{x\in X: d(x,x_0)<r\}$ in~$X$,
\begin{equation*}
        0 < \mu(2B) \le C \mu(B) < \infty.
\end{equation*}
Here and elsewhere we let $\delta B=B(x_0,\delta r)$.
A metric space with a doubling measure is proper\/
\textup{(}i.e.\ closed and bounded subsets are compact\/\textup{)}
if and only if it is complete.
See Heinonen~\cite{heinonen} for more on doubling measures.

A \emph{curve} is a continuous mapping from an interval,
and a \emph{rectifiable} curve is a curve with finite length.
We will only consider curves which are nonconstant, compact and rectifiable.
A curve can thus be parameterized by its arc length $ds$.
We follow Heinonen and Koskela~\cite{HeKo98} in introducing
upper gradients as follows (they called them very weak gradients).

\begin{deff} \label{deff-ug}
A nonnegative Borel function $g$ on $X$ is an \emph{upper gradient}
of an extended real-valued function $f$
on $X$ if for all nonconstant, compact and rectifiable curves
$\gamma: [0,l_{\gamma}] \to X$,
\begin{equation} \label{ug-cond}
        |f(\gamma(0)) - f(\gamma(l_{\gamma}))| \le \int_{\gamma} g\,ds,
\end{equation}
where we follow the convention that the left-hand side is $\infty$
whenever at least one of the 
terms therein is infinite.
If $g$ is a nonnegative measurable function on $X$
and if (\ref{ug-cond}) holds for \p-almost every curve (see below),
then $g$ is a \emph{\p-weak upper gradient} of~$f$.
\end{deff}

Here we say that a property holds for \emph{\p-almost every curve}
if it fails only for a curve family $\Ga$ with zero \p-modulus,
i.e.\ there exists $0\le\rho\in L^p(X)$ such that
$\int_\ga \rho\,ds=\infty$ for every curve $\ga\in\Ga$.
Note that a \p-weak upper gradient need not be a Borel function,
it is only required to be measurable.
On the other hand,
every measurable function $g$ can be modified on a set of measure zero
to obtain a Borel function, from which it follows that
$\int_{\gamma} g\,ds$ is defined (with a value in $[0,\infty]$) for
\p-almost every curve $\ga$.
For proofs of these and all other facts in this section
we refer to Bj\"orn--Bj\"orn~\cite{BBbook} and
Heinonen--Koskela--Shanmugalingam--Tyson~\cite{HKST}.

The \p-weak upper gradients were introduced in
Koskela--MacManus~\cite{KoMc}. It was also shown there
that if $g \in \Lploc(X)$ is a \p-weak upper gradient of $f$,
then one can find a sequence $\{g_j\}_{j=1}^\infty$
of upper gradients of $f$ such that $g_j-g \to 0$ in $L^p(X)$.
If $f$ has an upper gradient in $\Lploc(X)$, then
it has a \emph{minimal \p-weak upper gradient} $g_f \in \Lploc(X)$
in the sense that for every \p-weak upper gradient $g \in \Lploc(X)$ of $f$ we have
$g_f \le g$ a.e., see Shan\-mu\-ga\-lin\-gam~\cite{Sh-harm}
and Haj\l asz~\cite{Haj03}.
The minimal \p-weak upper gradient is well defined
up to a set of measure zero in the cone of nonnegative functions in $\Lploc(X)$.
Following Shanmugalingam~\cite{Sh-rev},
we define a version of Sobolev spaces on the metric measure space $X$.

\begin{deff} \label{deff-Np}
Let
\[
        \|f\|_{\Np(X)} = \biggl( \int_X |f|^p \, \dmu
                + \inf_g  \int_X g^p \, \dmu \biggr)^{1/p},
\]
where the infimum is taken over all upper gradients of $f$.
The \emph{Newtonian space} on $X$ is
\[
        \Np (X) = \{f: \|f\|_{\Np(X)} <\infty \}.
\]
\end{deff}
\medskip

The space $\Np(X)/{\sim}$, where  $f \sim h$ if and only if $\|f-h\|_{\Np(X)}=0$,
is a Banach space and a lattice, see Shan\-mu\-ga\-lin\-gam~\cite{Sh-rev}.
In this paper we assume that functions in $\Np(X)$ are defined everywhere,
not just up to an equivalence class in the corresponding function space.
For a measurable set $E\subset X$, the Newtonian space $\Np(E)$ is defined by
considering $(E,d|_E,\mu|_E)$ as a metric space on its own.
We say  that $f \in \Nploc(\Om)$ if
for every $x \in \Om$ there exists a ball $B_x\ni x$ such that
$B_x\subset\Om$ and $f \in \Np(B_x)$.
If $f,h \in \Nploc(X)$, then $g_f=g_h$ a.e.\ in $\{x \in X : f(x)=h(x)\}$,
in particular $g_{\min\{f,c\}}=g_f \chi_{\{f < c\}}$ for $c \in \R$.

\begin{deff} \label{deff-sobcap}
The \emph{Sobolev capacity} of an arbitrary set $E\subset X$ is
\[
\Cp(E) = \inf_u\|u\|_{\Np(X)}^p,
\]
where the infimum is taken over all $u \in \Np(X)$ such that
$u\geq 1$ on $E$.
\end{deff}

The capacity is countably subadditive.
We say that a property holds \emph{quasieverywhere} (q.e.)\
if the set of points  for which the property does not hold
has capacity zero.
The capacity is the correct gauge
for distinguishing between two Newtonian functions.
If $u \in \Np(X)$, then $u \sim v$ if and only if $u=v$ q.e.
Moreover, Corollary~3.3 in Shan\-mu\-ga\-lin\-gam~\cite{Sh-rev} shows that
if $u,v \in \Np(X)$ and $u= v$ a.e., then $u=v$ q.e.

A set $U\subset X$ is \emph{quasiopen} if for every
$\varepsilon>0$ there is an open set $G\subset X$ such that $\Cp(G)<\varepsilon$
and $G\cup U$ is open.
A function $u$ on a quasiopen set $U\subset X$ is \emph{quasicontinuous}
if for every $\varepsilon>0$ there is an open set $G\subset X$ such that $\Cp(G)<\varepsilon$ and $u|_{U \setm G}$ is finite and continuous.

\begin{deff} \label{def.PI.}
We say that $X$ supports a \emph{\p-Poincar\'e inequality} if
there exist constants $C>0$ and $\lambda \ge 1$
such that for all balls $B \subset X$,
all integrable functions $f$ on $X$ and all upper gradients $g$ of $f$,
\begin{equation} \label{PI-ineq}
        \vint_{B} |f-f_B| \,\dmu
        \le C (\diam B) \biggl( \vint_{\lambda B} g^{p} \,\dmu \biggr)^{1/p},
\end{equation}
where $ f_B
 :=\vint_B f \,\dmu
:= \int_B f\, d\mu/\mu(B)$.
\end{deff}

In the definition of Poincar\'e inequality we can equivalently assume
that $g$ is a \p-weak upper gradient---see the comments above.
If $X$ is complete and  supports a \p-Poincar\'e inequality
and $\mu$ is doubling, then Lipschitz functions
are dense in $\Np(X)$, see Shan\-mu\-ga\-lin\-gam~\cite{Sh-rev}.
Moreover, all functions in $\Np(X)$
and those in $\Np(\Om)$ are quasicontinuous,
see Bj\"orn--Bj\"orn--Shan\-mu\-ga\-lin\-gam~\cite{BBS5}.
This means that in the Euclidean setting, $\Np(\R^n)$ is the
refined Sobolev space as defined in
Heinonen--Kilpel\"ainen--Martio~\cite[p.~96]{HeKiMa},
see Bj\"orn--Bj\"orn~\cite{BBbook}
for a proof of this fact valid in weighted $\R^n$.
This is the main reason why, unlike in the classical Euclidean setting,
we do not need to
require the functions admissible in the definition of capacity to be $1$ in a
neighbourhood of $E$.

In Section~\ref{sect-fine-cont} the fine topology is
defined by means of thin sets, which in turn use the
variational capacity $\cp$.
To be able to define the variational capacity we first
need a Newtonian space with zero boundary values.
We let, for an arbitrary set $A \subset X$,
\[
\Np_0(A)=\{f|_{A} : f \in \Np(X) \text{ and }
        f=0 \text{ on } X \setm A\}.
\]
One can replace the assumption ``$f=0$ on $X \setm A$''
with ``$f=0$ q.e.\ on $X \setm A$''
without changing the obtained space
$\Np_0(A)$.
Functions from $\Np_0(A)$ can be extended by zero in $X\setm A$ and we
will regard them in that sense if needed.

\begin{deff}
Let $A\subset X$ be arbitrary. The \emph{variational
capacity} of $E\subset A$ with respect to $A$ is
\[
\cp(E,A) = \inf_u\int_X g_u^p\, d\mu,
\]
where the infimum is taken over all $u \in \Np_0(A)$
such that
$u\geq 1$ on $E$.
\end{deff}

\begin{remark} \label{rmk-cp}
The infimum above can equivalently be taken over
$u \in \Np(X)$ such that $u\ge1$ q.e.\ on $E$ and $u=0$ q.e.\ outside $A$.
We will call such functions admissible for the capacity $\cp(E,A)$.

Similarly, one can test the capacity $\Cp(E)$ by any
function $u \in \Np(X)$ such that $u \ge 1$ q.e.\ on $E$,
and we will call  such a function admissible for $\Cp(E)$.
\end{remark}

We will mainly be interested in the variational capacity
with respect to open sets $A$, but in Lemma~\ref{lemma-capfineclosure}
we will generalize an earlier result for the variational capacity
to arbitrary sets.
The variational capacity with respect to nonopen sets was recently
studied and used in
Bj\"orn--Bj\"orn~\cite{BBnonopen} and~\cite{BBvarcap}.
(Note that the roles of $A$ and $E$ are reversed in \cite{BBnonopen}
and \cite{BBvarcap} compared with this paper.)

\medskip

\emph{Throughout the rest of the paper, we assume that
$X$ is complete and  supports a \p-Poincar\'e inequality, and that
$\mu$ is doubling.}

\medskip

The following lemma from J.~Bj\"orn~\cite{Bj}
compares the capacities $\Cp$ and $\cp$, and the measure $\mu$.
Here and elsewhere, the letter $C$ denotes various positive
constants whose values may vary even within a line.

\begin{lemma} \label{lemma:capcomparison}
Let $E \subset B=B(x_0,r)$ with $0<r<\frac{1}{6}\diam(X)$. Then
\[
\frac{\mu(E)}{Cr^p}\le\cp(E,2B)\le\frac{C\mu(B)}{r^p}
\]
and
\[
\frac{\Cp(E)}{C(1+r^p)}\leq \cp(E,2B)\leq 2^{p}\biggl(1+\frac{1}{r^p}\biggr)\Cp(E).
\]
In particular,
\[
\frac{\mu(B)}{Cr^{p}}\le\cp(B,2B)\le \frac{C\mu(B)}{r^{p}}.
\]
\end{lemma}

We will also need the following result from
Bj\"orn--Bj\"orn--Shan\-mu\-ga\-lin\-gam~\cite{BBS5}.
(It was recently extended to arbitrary bounded sets $\Om$ in
Bj\"orn--Bj\"orn~\cite{BBvarcap}, but we will not need that generality here.)
Recall that $E \Subset \Om$ if $\itoverline{E}$ is a compact subset of $\Om$.

\begin{thm} \label{thm-outercap-cp}
Let $\Om \subset X$ be a bounded open set.
The variational capacity $\cp$
is an \emph{outer capacity} for sets $E \Subset \Om$,
i.e. 
\begin{equation} \label{eq-outercap-cp}
\cp(E,\Om)=\inf_{\substack{G \text{ open} \\  E\subset G \subset \Om}} \cp(G,\Om).
\end{equation}
\end{thm}

\section{Superminimizers and superharmonic functions}
\label{sect-superharm}

In this section we introduce superminimizers and superharmonic functions,
as well as obstacle problems, which all will be needed in later sections.
For further discussion and references on these topics
see Kinnunen--Martio~\cite{KiMa02} and~\cite{KiMasuperh}, and also
Bj\"orn--Bj\"orn~\cite{BBbook}
(which also contains proofs of the facts mentioned in this section,
but for Lemma~\ref{lem-localize}).

\begin{deff} \label{def-quasimin}
A function $u \in \Nploc(\Om)$ is a
\emph{\textup{(}super\/\textup{)}minimizer} in $\Om$
if
\[
      \int_{\{\phi \ne 0\}} g^p_u \, d\mu
           \le \int_{\{\phi \ne 0\}} g_{u+\phi}^p \, d\mu
           \quad \text{for all (nonnegative) } \phi \in \Np_0(\Om).
\]
A function $u$ is a \emph{subminimizer} if $-u$ is a superminimizer.
A \emph{\p-harmonic function} is a continuous minimizer.
\end{deff}

For characterizations of minimizers and superminimizers
  see A.~Bj\"orn~\cite{ABkellogg}.
Minimizers were first studied for functions in $\Np(X)$
in Shanmugalingam~\cite{Sh-harm}.
For a superminimizer $u$,  it was shown by
Kinnunen--Martio~\cite{KiMa02} that its \emph{lower semicontinuous regularization}
\begin{equation}
\label{essliminf}
 u^*(x):=\essliminf_{y\to x} u(y)= \lim_{r \to 0} \essinf_{B(x,r)} u
\end{equation}
is also a superminimizer  and $u^*= u$ q.e.
For an alternative proof of this fact see
Bj\"orn--Bj\"orn--Parviainen~\cite{BBP}.
If
$u$ is a minimizer, then $u^*$ is  continuous,
and thus \p-harmonic, see
Kinnunen--Shan\-mu\-ga\-lin\-gam~\cite{KiSh01}.

We will need the following weak Harnack inequalities.

\begin{thm}  \label{thm-sup-submin-obst-prob}
\textup{(Weak Harnack inequality for subminimizers)}
Let $q>0$.
Then there is $C>0$ such that
for all subminimizers $u$ in $\Om$  and all balls $B\subset 2B \subset \Om$,
\[
\esssup_{B} u
    \le C \biggl( \vint_{2B} u_\limplus^q \, \dmu \biggr)^{1/q}.
\]
Here $u_\limplus:=\max\{u,0\}$.
\end{thm}

\begin{thm} \label{theorem:WeakHarnack}
\textup{(Weak Harnack inequality for superminimizers)}
There are $q > 0$ and $C > 0$,
such that for all  nonnegative superminimizers $u$  in\/ $\Omega$,
\begin{equation} \label{eq:WeakHarnack}
     \biggl(\vint_{2B}u^{q}\, d\mu\biggr)^{1/q} \leq C\essinf_{B}u
\end{equation}
for every ball $B \subset 50 \la B \subset \Omega$.
\end{thm}

These Harnack inequalities were in metric spaces first obtained for minimizers
by Kinnunen--Shanmugalingam~\cite{KiSh01}, using De Giorgi's method,
whereas Kinnunen--Martio~\cite{KiMa02} soon afterwards modified
the arguments for sub- and superminimizers.
See Bj\"orn--Marola~\cite{BMarola}, p.\ 363, for some necessary
modifications of the statements in \cite{KiSh01} and \cite{KiMa02},
and for alternative proofs using Moser iteration.

For a nonempty bounded open set $G \subset X$ with $\Cp(X \setm G)>0$
we consider the following obstacle problem.
(If $X$ is
unbounded then the condition $\Cp(X \setm G)>0$ is of course
immediately fulfilled.)

\begin{deff} \label{deff-obst}
For $f \in \Np(G)$ and $\psi : G \to \eR$ let
\[
    \K_{\psi,f}(G)=\{v \in \Np(G) : v-f \in \Np_0(G)
            \text{ and } v \ge \psi \ \text{q.e. in } G\}.
\]
A function $u \in \K_{\psi,f}(G)$
is a \emph{solution of the $\K_{\psi,f}(G)$-obstacle problem}
if
\[
       \int_G g^p_u \, d\mu
       \le \int_G g^p_v \, d\mu
       \quad \text{for all } v \in \K_{\psi,f}(G).
\]
\end{deff}

A solution to the $\K_{\psi,f}(G)$-obstacle problem is easily
seen to be a superminimizer in $G$.
Conversely, a superminimizer $u$ in $\Om$ is a solution
of the $\K_{u,u}(G)$-obstacle problem for all open $G \Subset \Om$
with $\Cp(X\setm G)>0$.

If $\K_{\psi,f}(G) \ne \emptyset$, then
there is a solution of the $\K_{\psi,f}(G)$-obstacle problem,
and this solution is unique up to equivalence in $\Np(G)$.
Moreover, $u^*$
is the unique lower semicontinuously regularized solution.
If the obstacle $\psi$ is continuous, then $u^*$
is also continuous.
The obstacle $\psi$, as a continuous function, is even allowed
to take the value $-\infty$.
For $f \in \Np(G)$, we let $\oHpind{G} f$ denote the
continuous solution of the $\K_{-\infty,f}(G)$-obstacle problem;
this function is \p-harmonic in $G$ and has the same boundary values
(in the Sobolev sense) as $f$ on $\partial G$, and hence is also
called the solution of the Dirichlet problem with Sobolev boundary values.

\begin{deff} \label{deff-superharm}
A function $u : \Om \to (-\infty,\infty]$ is \emph{superharmonic}
in $\Om$ if
\begin{enumerate}
\renewcommand{\theenumi}{\textup{(\roman{enumi})}}%
\item $u$ is lower semicontinuous;
\item \label{cond-b}
$u$ is not identically $\infty$ in any component of $\Om$;
\item \label{cond-c}
for every nonempty open set $G \Subset \Om$ with $\Cp(X \setm G)>0$
and all functions
$v \in \Lip(X)$,
we have $\oHpind{G} v \le u$ in $G$
whenever $v \le u$ on $\bdy G$.
\end{enumerate}
\end{deff}

This definition of superharmonicity is equivalent to  the ones in
Hei\-no\-nen--Kil\-pe\-l\"ai\-nen--Martio~\cite{HeKiMa} and
Kinnunen--Martio~\cite{KiMa02}, see A.~Bj\"orn~\cite{ABsuper}.
A locally bounded superharmonic function is a superminimizer,
and all superharmonic functions are lower semicontinuously regularized.
Conversely, any lower semicontinuously regularized superminimizer
is superharmonic.

We will need the following comparison lemma for solutions to obstacle problems
from Bj\"orn--Bj\"orn~\cite{BB}.

\begin{lem} \label{lem-obst-le-reg}
\textup{(Comparison principle)}
Assume that\/ $\Om$ is bounded and such
that $\Cp(X \setm \Om)>0$.
Let $\psi_j \colon \Om \to \eR$ and $f_j \in \Np(\Om)$
be such that $\K_{\psi_j,f_j} \ne \emptyset$,
and let $u_j$ be the lower semicontinuously regularized
solution of the $\K_{\psi_j,f_j}$-obstacle problem, $j=1,2$.
If $\psi_1 \le \psi_2$ q.e.\ in\/ $\Om$
and\/ $(f_1-f_2)_\limplus \in \Np_0(\Om)$,
then $u_1 \le u_2$  in\/ $\Om$.
\end{lem}

The following simple localization lemma will be useful in the coming proofs.
For a proof in the metric space setting see  Farnana~\cite{FarFennAnn},
Lemma~3.6.

\begin{lem}  \label{lem-localize}
Let $u$ be the lower semicontinuously regularized solution of the $\K_{\psi,f}(\Om)$-obstacle
problem and let $\Om'\subset\Om$ be open.
Then $u$ is the lower semicontinuously regularized solution of the  $\K_{\psi,u}(\Om')$-obstacle
problem.
\end{lem}

\section{Fine topology}
\label{sect-fine-cont}

In this section we introduce the main concepts of this paper and present the necessary auxiliary results. At the end of the section, we prove
Theorem~\ref{thm-fine-quasicont-fine}.

A set $E\subset X$ is  \emph{thin} at $x\in X$ if
\begin{equation}   \label{deff-thin}
\int_0^1\biggl(\frac{\cp(E\cap B(x,r),B(x,2r))}{\cp(B(x,r),B(x,2r))}\biggr)^{1/(p-1)}
     \frac{dr}{r}<\infty.
\end{equation}
A set $U\subset X$ is \emph{finely open} if
$X\setminus U$ is thin at each point $x\in U$.
It is easy to see that the finely open sets give rise to a
topology, which is called the \emph{fine topology},
see Proposition~11.36 in Bj\"orn--Bj\"orn~\cite{BBbook}.
Every open set is finely open, but the converse is not true in general.

For any $E\subset X$,
the \emph{base} $b_p(E)$ is
the set of all points $x\in X$ so that $E$ is \emph{thick},
i.e.\ not thin, at $x$.
We also let $\clEp$ be the fine closure of $E$ and $\fineint E$ be the fine interior of $E$,
both taken with respect to the fine topology.

In the definition of thinness,
and in the sum \eqref{eq-sum-fine} below,
we make the convention that the integrand
is 1 whenever $\cp(B(x,r),B(x,2r))=0$.
This happens e.g.\ if $X=B(x,2r)$ is bounded,
but never e.g.\ if $r < \frac{1}{2}\diam X$.
Note that thinness is a local property, i.e.\
$E$ is thin at $x$ if and only if $E\cap B(x,\de)$ is thin at $x$,
where $\de >0$ is arbitrary.

\begin{deff}
A function $u : U \to \eR$, defined on a finely open set $U$, is
\emph{finely continuous}
if it is continuous when $U$ is equipped with the
fine topology and $\eR$ with the usual topology.
\end{deff}

Note that $u$ is finely continuous in $U$ if and only if
it is
finely continuous at every $x\in U$ in the sense that for all
$\eps>0$ there exists a finely open set $V\ni x$ such that
$|u(y)-u(x)|<\eps$ for all $y\in V$, if $u(x) \in \R$, and such that
$\pm u(y) > 1/\eps$ for all $y\in V$, if $u(x)=\pm \infty$,
or equivalently if and only if
the sets $\{x\in U:u(x)>k\}$ and $\{x\in U:u(x)<k\}$ are finely open
for all $k\in\R$.

Since every open set is finely open, 
the fine topology
generated by the
finely open\index{finely!open} sets is finer than the metric topology.
In fact, it is so fine that all superharmonic functions become
finely continuous. 
This is the content of the following theorem.

\begin{theorem}  \label{thm-superharm-p-fine}
Let $u$ be a superharmonic function in an open set\/ $\Om$.
Then $u$ is finely continuous in\/ $\Om$.
\end{theorem}

By Theorem~\ref{thm-coarsest-intro}, which we prove in 
Section~\ref{sect-coarsest}, the fine topology
is the coarsest topology with this property.
Together with its consequence Theorem~\ref{thm-qcont-fine} below,
Theorem~\ref{thm-superharm-p-fine} was obtained
by J.~Bj\"orn~\cite{JB-pfine}, Theorems~4.4 and~4.6,
and independently by Korte~\cite{korte08}, Theorem~4.3 and Corollary~4.4,
(they can also be found in Bj\"orn--Bj\"orn~\cite{BBbook} as
Theorems~11.38 and~11.40).

\begin{thm} \label{thm-qcont-fine}
Let $\Om$ be open.
Then every quasicontinuous function
$u:\Om\to \eR$ is finely continuous at q.e.\
$x\in\Om$.
In particular, this is true for all $u\in\Np\loc(\Om)$.
\end{thm}

At the end of this section (when proving Theorem~\ref{thm-fine-quasicont-fine}), we
extend the first part of the above result to finely open
and quasiopen sets.

We next  give some auxiliary lemmas.
The following characterization was essentially obtained in
Bj\"orn--Bj\"orn~\cite{BBnonopen}.

\begin{lem}\label{lem-fine-intclosure}
Let $E\subset X$ and $x\in X$. Then $x\in\fineint E$ if and only if $x\in E$ and $X\setminus E$ is thin at $x$. Moreover, we have
$$
\clEp=E\cup b_p(E).
$$
\end{lem}
\begin{proof}
For the characterization of fine interior points, see
Proposition~7.8 in \cite{BBnonopen}.
Accordingly, $x\notin\clEp$ if and only if $x$ is a fine interior point
of $X\setminus E$, i.e.\ $x\notin E$ and $E$ is thin at $x$.
Thus $x\in\clEp$ if and only if $x\in E$ or $E$ is thick at $x$.
\end{proof}

In Sections~\ref{sect-weak-Cartan} and \ref{sect-coarsest} we
will use the fact that the integral~\eqref{deff-thin} can be replaced
by a sum and the factor 2 in~\eqref{deff-thin}
can be replaced
by an arbitrary factor greater than 1. To prove this (see Lemma~\ref{lem-wienersum}), we need the following simple lemma, whose proof can be found e.g.\ in Bj\"orn--Bj\"orn~\cite{BBbook}, Lemma~11.22.

\begin{lem} \label{lem-cap-with-s-t}
Let $B=B(x_0,r)$ and $E\subset B$.
Then for every\/ $1<\tau<t<\tfrac{1}{4}\diam X$\textup,
$$
      \cp(E,tB) \le \cp(E,\tau B)
           \le C \biggl( 1 + \frac{t^p}{(\tau-1)^p} \biggr) \cp(E,tB).
$$
\end{lem}

\begin{lemma}
\label{lem-wienersum}
Let $E\subset X$, $x\in X$, $r_0>0$ and $\sigma >1$. Then $E$ is thin at $x$ if
and only if
\begin{equation} \label{eq-sum-fine}
\sum_{j=1}^{\infty}\biggl(\frac{\cp(E\cap B(x,\sigma^{-j}r_0),B(x,\sigma^{1-j}r_0))}
{\cp(B(x,\sigma^{-j}r_0),B(x,\sigma^{1-j}r_0))}\biggr)^{1/(p-1)}<\infty.
\end{equation}
\end{lemma}

\begin{proof}
Let $B_s=B(x,s)$ for $ s >0$.
Let $\rho < \tfrac{1}{8}\diam X$ and $\rho/\sigma \le r \le \rho$.
Then, by Lemma~\ref{lem-cap-with-s-t}
and the monotonicity of the capacity,
\[
   \frac{1}{C} \cp(E \cap B_{\rho/\sigma},B_{\rho})
   \le
\     \cp(E \cap B_r,B_{2r})
   \le C \cp(E \cap B_\rho,B_{\sigma \rho}),
\]
which together with the doubling property of $\mu$ and
Lemmas~\ref{lemma:capcomparison} and~\ref{lem-cap-with-s-t} shows that
\begin{align*}
\frac{1}{C}
          \biggl(\frac{\cp(E\cap B_{\rho/\sigma},B_{\rho})}{\cp(B_{\rho/\sigma},B_{\rho})}
     \biggr)^{1/(p-1)}
     &\le              \int_{\rho/\sigma}^\rho
          \biggl(\frac{\cp(E\cap B_r,B_{2r})}{\cp(B_r,B_{2r})}\biggr)^{1/(p-1)}
   \frac{dr}{r} \\
   & \le  C\biggl(\frac{\cp(E\cap B_\rho,B_{\sigma\rho})}{\cp(B_\rho,B_{\sigma\rho})}
     \biggr)^{1/(p-1)}.
\end{align*}
Hence \eqref{deff-thin} converges if and only if \eqref{eq-sum-fine} converges.
\end{proof}

\begin{lemma} \label{lem-HeKiMa-12.11}
Let $E\subset X$ be thin at $x\in\itoverline{E}\setminus E$. Then there is an
open neighbourhood $G$ of $E$ such that $G$ is thin at $x$ and $x\notin G$.
\end{lemma}

\begin{proof} Let $B_j=B(x,2^{-j})$, $j=1,2,\ldots$\,. By Lemma~\ref{lem-cap-with-s-t},
\[
\cp(E \cap \itoverline{B}_j ,2B_j) \le C \cp(E \cap \itoverline{B}_j ,4B_j)
    \le C \cp(E \cap 2B_j ,4B_j).
\]
Since the variational capacity is an outer capacity, by Theorem~\ref{thm-outercap-cp},
we can find open sets $G_j \supset E \cap \itoverline{B}_j $ such that
\[
   \biggl( \frac{\cp(G_j,2B_j)}{\cp(B_j,2B_j)}\biggr)^{1/(p-1)}\le
\biggl( \frac{\cp(E \cap \itoverline{B}_j,2B_j)}{\cp(B_j,2B_j)}\biggr)^{1/(p-1)} + 2^{-j}.
\]
Let
\[
     G = (X \setm \itoverline{B}_1) \cup (G_1 \setm \itoverline{B}_2)
        \cup           ((G_1 \cap G_2) \setm \itoverline{B}_3)
        \cup           ((G_1 \cap G_2  \cap G_3) \setm \itoverline{B}_4)
        \cup \ldots.
\]
Then $G$ is open and contains $E$, and $x \notin G$.
Moreover $G \cap B_j \subset G_j$ and thus, by combining the estimates and using
Lemmas~\ref{lemma:capcomparison} and~\ref{lem-wienersum},
\[
\sum_{j=1}^{\infty}\biggl(\frac{\cp(G\cap B_j,2B_j)}{\cp(B_j,2B_j)}\biggr)^{1/(p-1)}
\le  C \sum_{j=1}^{\infty}\biggl(\frac{\cp(E \cap 2B_j ,4B_j)}
{\cp(2B_j,4B_j)}\biggr)^{1/(p-1)} +1<\infty.
\]
Hence the claim follows from Lemma~\ref{lem-wienersum}.
\end{proof}

Theorem~\ref{thm-qcont-fine} can be used to prove
the following generalization of Corollary~4.5 in
J.~Bj\"orn~\cite{JB-pfine} (which can also be found as Corollary~11.39 in \cite{BBbook}),
where \eqref{cap-p-fine-closure-A} was obtained for bounded open $A$
with $\Cp(X \setm A)>0$ and $E \Subset A$.
There is also an intermediate version in
Bj\"orn--Bj\"orn~\cite{BBvarcap}, Corollary~4.7.
In \cite{JB-pfine}, Corollary~4.5 was used to obtain Theorem~\ref{thm-qcont-fine}.
Here we instead use Theorem~\ref{thm-qcont-fine} to obtain
Lemma~\ref{lemma-capfineclosure}, i.e.\ to improve Corollary~4.5
from~\cite{JB-pfine}.

\begin{lemma}\label{lemma-capfineclosure}
If $E\subset X$, then
$
\Cp(\clEp)= \Cp(E).
$
Moreover, if $E\subset A$, then
\begin{equation}
\cp(E,A) = \cp(\clEp \cap A,A).
\label{cap-p-fine-closure-A}
\end{equation}
If furthermore $\cp(E,A)<\infty$, then $\Cp(\clEp \setm \fineint A)=0$ and
\begin{equation}
\cp(E,A) = \cp(\clEp \cap A,A) = \cp(\clEp \cap \fineint A, \fineint A).
\label{cap-p-fine-closure}
\end{equation}
\end{lemma}

\begin{proof} The inequality $\Cp(\clEp)\le\Cp(E)$ follows since any
$v\in N^{1,p}(X)$ admissible for the capacity $\Cp(E)$ is also admissible
for the capacity $\Cp(\clEp)$.
Indeed, if $x\in\clEp$ is a fine continuity point of $v$, then
\[
v(x)=\finelim_{y\to x}v(y)\ge 1.
\]
Since q.e.\ point in $X$ is a fine continuity point for $v\in N^{1,p}(X)$,
by Theorem~\ref{thm-qcont-fine},
we conclude that $v\ge 1$ q.e.\ in  $\clEp$.
The converse inequality is trivial.

Similarly, if $u\in\Np_0(A)$ is admissible for $\cp(E,A)$ then $u\ge1$
q.e.\ in $\clEp$ and $u=0$ q.e.\ in $X\setm\fineint A$.
This proves the nontrivial inequality
in \eqref{cap-p-fine-closure},  and also
that  $\Cp(\clEp \setm \fineint A)=0$ if there exists such a $u$.
Finally, \eqref{cap-p-fine-closure-A}
is trivial if $\cp(E,A)=\infty$.
\end{proof}

As a main consequence of Lemma~\ref{lemma-capfineclosure}, 
we end this section by proving Theorem~\ref{thm-fine-quasicont-fine}.

\begin{proof}[Proof of Theorem~\ref{thm-fine-quasicont-fine}.]
(a) For each $j=1,2,\ldots$, find an open set $G_j$ with
$\Cp(G_j)<2^{-j}$ so that $U\cup G_j$  is open.
By Lemma~\ref{lemma-capfineclosure}, we have $\Cp(\clGjp)=\Cp(G_j)<2^{-j}$.
Let $E:=U\cap\bigcap_{j=1}^\infty \clGjp$.
Then $\Cp(E)=0$. Moreover,
\[
    V_j:=U \setm \clGjp = (U \cup G_j) \setm \clGjp
\]
is finely open, and thus $V:=\bigcup_{j=1}^\infty V_j = U \setm E$ is finely open.

(b) By (a) 
we may assume that $U=V \cup E$, where $V$ is finely open and $\Cp(E)=0$.
As $u$ is quasicontinuous, we can for each $j=1,2,\ldots$
find an open set $G_j$ with $\Cp(G_j)<2^{-j}$
so that $u|_{V\setminus G_j}$ is continuous.
By Lemma~\ref{lemma-capfineclosure}, we have
$\Cp(\clGjp)=\Cp(G_j)<2^{-j}$. Hence the set
\[
A:=E \cup \biggl(V\cap\bigcap_{j=1}^\infty \clGjp \biggr)
\]
is of capacity zero.
If $x\in U\setminus A$,
then $x$ belongs to the finely open set
$V\setminus \clGkp$ for some $k$,
and the fine continuity of $u$ at $x$ follows from the continuity of
$u|_{V\setm G_k}$ since
the fine topology is finer than the metric topology.
\end{proof}

\section{The weak Cartan property}
\label{sect-weak-Cartan}

Our aim in this section is to obtain the following
\emph{weak Cartan property}.

\begin{thm}\textup{(Weak Cartan property)} \label{thm-weak-Cartan}
Assume that $E$ is thin at $x_0\notin E$. Then there exist  a ball $B$ centred at $x_0$
and superharmonic functions $u,u'\in\Np(B)$ such that
\[
0\le u\le 1,\quad  0\le u'\le 1, \quad
u(x_0)<1, \quad u'(x_0)<1 \quad \text{and}  \quad
E\cap  B \subset F\cup F',
\]
where $F=\{x\in B:u(x)=1\}$ and $F'=\{x\in B:u'(x)=1\}$.

In particular, with $v=\max\{u,u'\}$ we have $v(x_0)<1$ and $v=1$ in
$E\cap B$.
\end{thm}

Note that $u$, $u'$ and $v$ above are
lower semicontinuous, quasicontinuous and finely continuous in $B$.
In the proof
we will use
two lemmas which are also of independent interest
(see e.g.\ the proof of Proposition~\ref{prop-pt-thick}).
We shall frequently use the following notion.

\begin{deff}
We say that a function $u$ is the \emph{capacitary potential}
of a set $E$ in $B\supset E$ if it is the
lower semicontinuously regularized solution of the $\K_{\chi_E,0}(B)$-obstacle problem.
\end{deff}

\begin{lem}    \label{lem-est-pot-bdryB}
Let $B=B(x_0,r)$ and $B_0$ be balls such that $50\lambda B\subset B_0$ and
$\Cp(X\setm B_0)>0$. Also let $E\subset \tfrac12B_0$ be such that
$E\cap \bigl(2B\setm\tfrac12B\bigr)=\emptyset$ and let $u$ be the capacitary potential of $E$ in $B_0$.
Then
\begin{equation}   \label{eq-est-pot-bdryB}
\sup_{\bdry B} u \le C' \biggl( \frac{\cp(E,B_0)}{\cp(B,B_0)}
\biggr)^{1/(p-1)}.
\end{equation}
\end{lem}

\begin{proof}
Let $m=\inf_{B} u$.
If $m=0$, then the left-hand side in \eqref{eq-est-pot-bdryB} is $0$,
by Theorem~\ref{theorem:WeakHarnack}, and \eqref{eq-est-pot-bdryB} follows.
If $m=1$, then $\cp(B,B_0) \le \int_{B_0} g_u^p \, d\mu = \cp(E,B_0)$
and \eqref{eq-est-pot-bdryB} holds for any $C' \ge 1$.
Assume therefore that  $0 < m<1$.
Thus
the functions
\[
u_1=\min\Bigl\{\frac{u}{m},1\Bigr\}\quad\textrm{and}\quad u_2=\frac{u-mu_1}{1-m}
\]
are admissible in the definition of $\cp(B,B_0)$ and $\cp(E,B_0)$,
respectively, (in view of Remark~\ref{rmk-cp}).
Note that for a.e.\ $x\in B_0$, at least one of $g_{u_1}(x)$ and $g_{u_2}(x)$ vanishes.
As $u$ is the capacitary potential of $E$ in $B_0$, we therefore
obtain that
\begin{align*}
\cp(E,B_0) & = \int_{\{u \le m\}} g_u^p\,d\mu + \int_{\{u>m\}} g_u^p\,d\mu \\
&= m^p \int_{B_0} g_{u_1}^p\,d\mu
      + (1-m)^p \int_{B_0} g_{u_2}^p\,d\mu \\
&\ge m^p \cp(B,B_0) + (1-m)^p \cp(E,B_0).
\end{align*}
It follows that
\[
\cp(B,B_0) \le \frac{1-(1-m)^p}{m^p} \cp(E,B_0)
  \le p m^{1-p} \cp(E,B_0)
\]
and equivalently,
\begin{equation} \label{eq-wCartlem-1}
m \le \biggl( \frac{p \cp(E,B_0)}{\cp(B,B_0)} \biggr)^{1/(p-1)}.
\end{equation}
Now, let $B'=B(x',r')$ be such that $x'\in \bdry B$, $r'=\frac{1}{5}r$, and
$
\sup_{\bdry B} u \le \sup_{B'} u.
$
Then $2B'\Subset2B\setm\tfrac12\itoverline{B}$ and as $u$ is a nonnegative
lower semicontinuously regularized
minimizer in $2B\setm\tfrac12\itoverline{B}$, the weak Harnack inequality
for subminimizers (Theorem~\ref{thm-sup-submin-obst-prob})
implies that for every $q>0$, there exists a constant $C_q$,
independent of $u$ and $B'$,
such that
\begin{equation} \label{eq-wCartlem-2}
\sup_{B'} u \le C_q \biggl( \vint_{2B'} u^q\,d\mu \biggr)^{1/q}.
\end{equation}
Finally, as $u$ is a superminimizer in $B_0$,
the weak Harnack inequality for superminimizers
(Theorem~\ref{theorem:WeakHarnack}) and the doubling property of $\mu$
imply that for some $q>0$ and $\Ct>0$,
independent of $u$, $B$ and $B'$,
\begin{equation} \label{eq-wCartlem-3}
m = \inf_B u \ge \Ct \biggl( \vint_{2B} u^q\,d\mu \biggr)^{1/q}
\ge C \biggl( \vint_{2B'} u^q\,d\mu \biggr)^{1/q}.
\end{equation}
Combining \eqref{eq-wCartlem-1}--\eqref{eq-wCartlem-3} gives
\eqref{eq-est-pot-bdryB}.
\end{proof}

\begin{remark}  \label{rem-converse-with-inf}
Lemma~3.9 in J.~Bj\"orn~\cite{JB-pfine} (Lemma~11.20 in~\cite{BBbook} or Lemma~5.6 in
Bj\"orn--MacManus--Shanmugalingam~\cite{BMS} in linearly locally
connected spaces)
provides us with the converse inequality to~\eqref{eq-est-pot-bdryB}, viz.\
\begin{equation}   \label{eq-est-pot-bdryB-inf}
\inf_{\bdry B} u \ge C'' \biggl( \frac{\cp(E,B_0)}{\cp(B,B_0)}
\biggr)^{1/(p-1)}.
\end{equation}
\end{remark}

\begin{prop}  \label{prop-u(x0)-inf-product}
For a ball $B=B(x_0,r)$ with $\Cp(X\setm B)>0$ let $B_j=\s^{-j}B$,
$j=0,1,\ldots$, where $\s\ge50\la$ is fixed.
Assume that $E\subset \tfrac12B$ is such that
$E\cap \bigl(2B_j\setm\tfrac12B_j\bigr)=\emptyset$ for all $j=0,1,\ldots$, and let
$u$ be the capacitary potential of $E$ in $B$. Then
\[
1 - \prod_{j=0}^\infty (1-ca_j) \le u(x_0) \le 1 - \prod_{j=0}^\infty (1-a_j),
\]
where
\[
a_j = \min\biggl\{ 1, C' \biggl( \frac{\cp\bigl(E\cap \tfrac12B_{j},B_j\bigr)}
    {\cp(B_{j+1},B_j)}
   \biggr)^{1/(p-1)} \biggr\},
\]
$c=C''/C'>0$ and $C'$ and $C''$ are as in~\eqref{eq-est-pot-bdryB}
and~\eqref{eq-est-pot-bdryB-inf}.
\end{prop}

\begin{remark} \label{rem-Wolff}
(a) The case $c\ge 1$ is not excluded in Proposition \ref{prop-u(x0)-inf-product}. However, by \eqref{eq-est-pot-bdryB} and \eqref{eq-est-pot-bdryB-inf}, the case $c>1$ holds true only if $a_j=0$ for all $j=0,1,\dots$\,.
By \eqref{eq-est-def-aj2}, the case $c=1$ holds true only if
\[
\inf_{\bdry B_{j+1}} u_j =\sup_{\bdry B_{j+1}} u_j
\]
for all $j=0,1,\dots$\,. See the proof below for the notation here.

(b) The first inequality in Proposition~\ref{prop-u(x0)-inf-product}
can be obtained from Lemma~5.7 in
Bj\"orn--MacManus--Shanmugalingam~\cite{BMS} (in linearly locally
connected spaces) or from Proposition~3.10 in J.~Bj\"orn~\cite{JB-pfine}
(alternatively Theorem~11.21 in \cite{BBbook}).
In this paper we will not need it, 
but we have chosen to include it here as
the proof below shows that both
inequalities can be obtained simultaneously.

In fact, by taking logarithms, the left estimate in
Proposition~\ref{prop-u(x0)-inf-product} implies
\[
1 -u(x_0) \le \exp\biggl( -c \sum_{j=0}^\infty a_j \biggr),
\]
which in particular shows that if $E$ is thick at $x_0$ then $u(x_0)=1$.
As for the right estimate in Proposition~\ref{prop-u(x0)-inf-product},
it is easily shown by induction that
$1 - \prod_{j=0}^n (1-a_j) \le \sum_{j=0}^n a_j$ and hence we obtain the
qualitative estimate
\begin{equation}   \label{eq-wolff-pot-est}
u(x_0) \le C' \sum_{j=0}^\infty \biggl(
\frac{\cp\bigl(E\cap \tfrac12B_{j},B_j\bigr)}{\cp(B_{j+1},B_j)}\biggr)^{1/(p-1)},
\end{equation}
which in $\R^n$, with $p<n$, 
reduces to a special case of the estimate
(6.1) in Maz\cprime ya--Havin~\cite{MazHa72}.
 It corresponds to the Wolff potential estimates for superharmonic functions
in e.g.\ Kilpel\"ainen--Mal\'y~\cite{KiMa}, Mikkonen~\cite{Mikkonen} and
Bj\"orn--MacManus--Shanmugalingam~\cite{BMS} and partly generalizes
Theorem~3.6 in J.\ Bj\"orn~\cite{JBCalcVar}.
More precisely, the Wolff potential for the capacitary measure of $E$
is easily seen to be comparable to the sum
in~\eqref{eq-wolff-pot-est}.
The estimates for general superharmonic functions in \cite{KiMa},
\cite{Mikkonen} and \cite{BMS} contain an additional term, such as
$(\vint_{B_0}u^p\,d\mu)^{1/p}$, but since the
potential $u$ has boundary values $0$ on $\bdry B_0$, this term can be
avoided in this case, cf.~\cite[Theorem~3.6]{JBCalcVar}.

In particular,~\eqref{eq-wolff-pot-est} implies the necessity part of
the Wiener criterion in certain domains (such that
$(2B_j\setm\tfrac12B_j)\setm\Om)=\emptyset$ for all sufficiently large
$j$ and some $\sigma>0$), since for a sufficiently small ball $B=B(x_0,r)$, the
capacitary potential of $\tfrac12 B\setm\Om$ in $B$ will not attain
its boundary value $1$ at $x_0$.
Note that the necessity part of the Wiener criterion is still
open for \p-harmonic functions (based on upper gradients) in
metric spaces.
\end{remark}

\begin{proof}
For $j=0,1,\ldots$, let $u_j$ be the capacitary potential of
$E_j=E\cap \tfrac12B_{j}$ in $B_j$.
Then $u=u_0$.
Lemma~\ref{lem-est-pot-bdryB} and Remark~\ref{rem-converse-with-inf}
imply that for all $j=0,1,\ldots$,
\begin{equation}    \label{eq-est-def-aj2}
ca_j \le \inf_{\bdry B_{j+1}} u_j \le \sup_{\bdry B_{j+1}} u_j \le a_j.
\end{equation}
We shall show by induction that for all $k=1,2,\ldots$,
\begin{equation}  \label{eq-step-MI2}
1-\sup_{\bdry B_{k}} u \ge \prod_{j=0}^{k-1} (1-a_j) =: b_k
\quad \text{and} \quad
1-\inf_{\bdry B_{k}} u \le \prod_{j=0}^{k-1} (1-ca_j) =: b'_k.
\end{equation}
By~\eqref{eq-est-def-aj2}, this clearly holds for $k=1$.
Assume that \eqref{eq-step-MI2} holds for some $k\ge1$ and let $G_k=\{x\in B_k:u(x)>1-b_k\}$.
Then $G_k$ is open by the lower semicontinuity of $u$, and
since $\sup_{\bdry B_{k}}u\le 1-b_k$, we have
$v_k:=(u-(1-b_k))_\limplus\in\Np_0(G_k)$.
Lemma~\ref{lem-localize} shows that $v_k$ is the
lower semicontinuously regularized solution of the $\K_{\psi_k,0}(G_k)$-obstacle problem,
where $\psi_k=(\chi_{E_0}-(1-b_k))_\limplus=b_k\chi_{E_k}$ in $B_k$.

On the other hand, by the minimum principle for superharmonic functions,
we have $u\ge1-b'_k$ in $B_k$ and Lemma~\ref{lem-localize} again shows that
$v'_k:=u-(1-b'_k)\ge0$ is the lower semicontinuously regularized solution
of the $\K_{\psi'_k,v'_k}(B_k)$-obstacle problem,
where $\psi'_k=(\chi_{E_0}-(1-b'_k))_\limplus=b'_k\chi_{E_k}$ in $B_k$.

Since $0\le u_k\in\Np_0(B_k)$ is the lower semicontinuously
regularized solution of the $\K_{\chi_{E_k},0}(B_k)$-obstacle problem,
the comparison principle
(Lemma~\ref{lem-obst-le-reg}) yields that
$v'_k\ge b'_k u_k$ in $B_k$ and that
$v_k\le b_k u_k$ in $G_k$, and hence in $B_k$.
In particular, by~\eqref{eq-est-def-aj2},
\[
\sup_{\bdry B_{k+1}} v_k \le \sup_{\bdry B_{k+1}} b_k u_k \le a_k b_k
\quad \text{and} \quad
\inf_{\bdry B_{k+1}} v'_k \ge \inf_{\bdry B_{k+1}} b'_k u_k \ge ca_k b'_k.
\]
Hence
\begin{align*}
\sup_{\bdry B_{k+1}} u &\le \sup_{\bdry B_{k+1}} v_k + 1-b_k
\le a_k b_k +1-b_k = 1- b_k(1-a_k)=1- b_{k+1}  \\
\intertext{and}
\inf_{\bdry B_{k+1}} u &= \inf_{\bdry B_{k+1}} v'_k + 1-b'_k
\ge ca_k b'_k +1-b'_k =1-b'_k(1-ca_k)=1- b'_{k+1},
\end{align*}
which proves~\eqref{eq-step-MI2} for $k+1$.
By induction, \eqref{eq-step-MI2} holds for all $k=1,2\ldots$\,.
Since $u$ is lower semicontinuously regularized, letting $k\to\infty$ gives
\begin{align*}
u(x_0) &= \liminf_{x\to x_0} u(x) \le 1 -\lim_{k\to\infty} b_k
= 1 - \prod_{j=0}^\infty (1-a_j)
\end{align*}
and, by the minimum principle,
\begin{align*}
u(x_0) &\ge 1 -\lim_{k\to\infty} b'_k = 1 - \prod_{j=0}^\infty (1-ca_j).
\qedhere
\end{align*}
\end{proof}

We are now ready to prove the weak Cartan property. The proof uses a separation argument which has been inspired by
Theorem~3.2 in Heinonen--Kilpel\"ainen--Martio~\cite{HeKiMa89},
and whose idea goes back to 
Lindqvist--Martio~\cite{Lind-Mar}.

\begin{proof}[Proof of Theorem~\ref{thm-weak-Cartan}]
By Lemma~\ref{lem-HeKiMa-12.11}, we can assume that $E$
is open.
For $r>0$ let $B_j=\s^{-j}B(x_0,r)$ with $\s=50\la$ be as in
Proposition~\ref{prop-u(x0)-inf-product}.
Also let $D_j=\bigl(\tfrac12 B_j\setm2\itoverline{B}_{j+1}\bigr)\cap E$ and
$E_j=\bigcup_{i=j}^\infty D_i$, $j=0,1,\ldots$\,.
Note that $E_0\cap \bigl(2B_j\setm\tfrac12B_j\bigr)=\emptyset$ for all $j=0,1,\ldots$\,.
Proposition~\ref{prop-u(x0)-inf-product} then implies that the capacitary potential
$u$ of $E_0$ in $B_0=B(x_0,r)$ satisfies
\[
u(x_0)\le 1 - \prod_{j=0}^\infty (1-a_j),
\]
where
\[
a_j = \min\biggl\{ 1, C' \biggl( \frac{\cp(E_{j},B_j)}{\cp(B_{j+1},B_j)}
    \biggr)^{1/(p-1)} \biggr\}
\]
and $C'$ is as in Lemma~\ref{lem-est-pot-bdryB}.
Since $E$ is thin at $x_0$, we can find $r>0$ so that all $a_j\le\tfrac12$
and $\sum_{j=0}^{\infty}a_j<\infty$
(by Lemma~\ref{lem-wienersum}).
Hence the series $\sum_{j=0}^{\infty}\log(1-a_j)$ converges as well,
which implies that $\prod_{j=0}^\infty (1-a_j)>0$, i.e.\ that $u(x_0)<1$.
On the other hand, we have $u=1$ in $E_0$, as $E_0$ is open.

Similarly, since
$2\itoverline B_j\setm\tfrac12B_j
\subset \tfrac{1}{5}\bigl(\tfrac12B_{j-1}\setm2\itoverline B_j\bigr)$,
replacing $r$ by $r'=\tfrac{1}{5}r$ in the above argument
provides us with the capacitary potential $u'$ in $B(x_0,r')$
which satisfies $u'(x_0)<1$ and $u'=1$ in
$\bigl(E\cap B\bigl(x_0,\tfrac12 r'\bigr)\bigr)\setm E_0$.
Letting $B=B\bigl(x_0,\tfrac{1}{2}r'\bigr)$ concludes the proof.
\end{proof}

We end this section by proving Theorem~\ref{thm-j}.

\begin{proof}[Proof of Theorem~\ref{thm-j}]
\ref{j-b}  \eqv \ref{j-a} \eqv \ref{j-aa}
This follows directly from Lemma~\ref{lem-fine-intclosure}.

\ref{j-b} \imp \ref{j-c}
This follows from the weak Cartan property (Theorem~\ref{thm-weak-Cartan}).

\ref{j-c} \imp \ref{j-k} This is trivial.

\ref{j-k} \imp \ref{j-a}
We can find  $\de$  and a ball $B \ni x_0$
such that  $v(x_0) < \de <    v(x)$ for all $x \in B \cap E$.
As $v$ is finely continuous, by Theorem~\ref{thm-superharm-p-fine},
$V:=\{x \in B: v(x) < \de\}$ is a finely open fine neighbourhood of $x_0$.
Since $E \cap V = \emptyset$, we see that
$x_0 \notin \clEp$.
\end{proof}

\section{Consequences of the weak Cartan property}
\label{sect-coarsest}

In this section we establish several consequences of the weak Cartan property.
First, we prove Theorem~\ref{thm-coarsest-intro}, i.e.\
that the fine topology is the coarsest
topology making all superharmonic functions continuous, and that the base
of its neighbourhoods is given by finite intersections of level sets of
superharmonic functions.

The coarsest topology related to Theorem~\ref{thm-coarsest-intro}
is traditionally formulated using global superharmonic functions
on $\R^n$. This definition relies on the following extension result:
If $u$ is superharmonic in $\Om \subset \R^n$ and $G \Subset \Om$,
then there is a superharmonic function $v$ on $\R^n$ such that $v=u$ in $G$,
see  Theorem~3.1 in Kilpel\"ainen~\cite{Ki89} (for unweighted $\R^n$)
and Theorem~7.30 in Hei\-no\-nen--Kil\-pe\-l\"ai\-nen--Martio~\cite{HeKiMa}
(for weighted $\R^n$).
Such an extension result is not known for unbounded metric spaces,
while it is false for bounded metric spaces as
there are only constant superharmonic functions on $X$ if $X$ is bounded.
Therefore we directly prove the following local formulation.

\begin{thm}\label{thm-coarsest}
A set $U\subset X$ is a fine neighbourhood of $x_0$ if and only if there exist
constants $c_j$ and bounded superharmonic functions $u_j$ in some
ball $B\ni x_0$, $j=1,2,\ldots,k$, such that
\begin{equation}   \label{eq-nbhd-base}
x_0 \in \bigcap_{j=1}^k \{ x\in B: u_j(x)<c_j \} \subset U.
\end{equation}
\end{thm}

The proof shows that the neighbourhood base condition always holds with $k=2$.
Recall that a set $U$  is a \emph{fine neighbourhood}
of a point $x_0$ if it contains
a  finely open set $V \ni x_0$; it is not
required
that $U$ itself is finely open.

\begin{proof} Let $U\subset X$. First, we assume that there exist constants $c_j$ and bounded superharmonic functions $u_j$
in a ball $B\ni x_0$, $j=1,2,\ldots,k$, such that \eqref{eq-nbhd-base} holds.
By Theorem~\ref{thm-superharm-p-fine}, each $u_j$ is finely
continuous and hence
\[
V_j:= \{ x\in B: u_j(x)<c_j \}
\]
is finely open.
It follows that $\bigcap_{j=1}^k V_j$ is finely open and hence
$U$ is a fine neighbourhood of $x_0$.

To prove the converse, let $E=X\setm U$.
Then $x\notin E$ and $E$ is thin at $x$.
Let $B$, $F$, $F'$, $u$, and $u'$ be as given by
the weak Cartan property (Theorem~\ref{thm-weak-Cartan}).
Then
\[
B\cap U = B \setm E \supset B\setm (F \cup  F')
=\{x\in B:u(x)<1\}\cap\{x\in B:u'(x)<1\},
\]
i.e.\  the fine neighbourhood base condition holds with $k=2$.
\end{proof}

\begin{proof}[Proof of Theorem~\ref{thm-coarsest-intro}.]
By Theorem~\ref{thm-superharm-p-fine}, the fine topology makes all
superharmonic functions on all open subsets of $X$ continuous.
To show that it is the coarsest topology with this
property, let $\T$ be such a topology on $X$, and let $U\subset X$
be finely open.
We shall show that for every $x_0\in U$ there exists $V\in\T$ such that
$x_0\in V\subset U$.
Indeed, let $u_1$ and $u_2$ be the superharmonic functions provided by
Theorem~\ref{thm-coarsest} and so that~\eqref{eq-nbhd-base} holds.
Since $\T$ makes all superharmonic functions continuous, we get that
the level sets $\{ x\in B: u_j(x)<c_j \}$ belong to $\T$, and so does their
intersection.
In view of~\eqref{eq-nbhd-base} this concludes the proof.
\end{proof}

Note that here it is not enough to only consider all superharmonic
functions on $X$, as these may be just the constants (if $X$ is bounded).
Therefore, superharmonic functions on all open sets (or balls) in $X$ have to
be considered in Theorem~\ref{thm-coarsest-intro}.

As a consequence of Proposition~\ref{prop-u(x0)-inf-product}
we can also deduce Proposition~\ref{prop-pt-thick}.

\begin{proof}[Proof of Proposition~\ref{prop-pt-thick}]
Let $\s=50\la$, $E=\{x_0\}$, $B=B(x_0,r)$, $B_j$ and $u$ be as in
Proposition~\ref{prop-u(x0)-inf-product}.
Since $\Cp(\{x_0\})>0$, we have $u(x_0)=1$.
Proposition~\ref{prop-u(x0)-inf-product} yields
\[
u(x_0) \le 1 - \prod_{j=0}^\infty (1-a_j),
\]
where
\[
a_j = \min\biggl\{ 1, C' \biggl( \frac{\cp(\{x_0\},B_j)}
    {\cp(B_{j+1},B_j)}
   \biggr)^{1/(p-1)} \biggr\}
\]
and $C'$ is as in Lemma~\ref{lem-est-pot-bdryB}.
If $E$ were thin at $x_0$, we could find $r>0$ so that all $a_j\le\tfrac12$
and $\sum_{j=0}^{\infty}a_j<\infty$
(by Lemma~\ref{lem-wienersum}).
Hence the series $\sum_{j=0}^{\infty}\log(1-a_j)$ would converge as well,
implying that $\prod_{j=0}^\infty (1-a_j)>0$, i.e.\ that $u(x_0)<1$,
which is a contradiction. Thus $\{x_0\}$ is thick at $x_0$.
\end{proof}

The proof of the following lemma has been inspired by the proof of
Lemma~12.24 in Heinonen--Kilpel\"ainen--Martio~\cite{HeKiMa},
but here we make use
of the weak Cartan property to simplify the argument.

\begin{lem}  \label{lem-thin-Cp-to-0}
If a set $E$ is thin at $x_0$ then for every ball $B\ni x_0$
\[
\lim_{\rho\to0} \cp(E\cap B(x_0,\rho),B) =0.
\]
\end{lem}

\begin{proof} Without loss of generality we may assume that
$\diam B < \frac{1}{6} \diam X$.
Since the variational capacity
is an outer capacity, by Theorem~\ref{thm-outercap-cp}, we
see that
\[
\cp(E\cap B(x_0,\rho),B) \le \cp(B(x_0,\rho),B) \to \cp(\{x_0\},B),
\quad \text{as } \rho\to 0,
\]
 and thus the result is trivial if $\cp(\{x_0\},B)=0$.
If $x_0\in E$ and $\cp(\{x_0\},B)>0$, then $\Cp(\{x_0\})>0$,
by Lemma~\ref{lemma:capcomparison}.
Proposition~\ref{prop-pt-thick} then
implies that $E$ is thick at $x_0$, a contradiction.
We can therefore assume that $x_0\notin E$ and
$\cp(\{x_0\},B)>0$.

Let $0<\eps<\cp(\{x_0\},B)$ be arbitrary.
By the weak Cartan property (Theorem~\ref{thm-weak-Cartan}), there
exist a ball $B'\subset 2B' \subset B$, containing $x_0$,
and $v\in\Np(B')$ such that $v(x_0)<1$ and
$v=1$ in $E\cap B'$.
Since $v \in \Np(B')$ it is quasicontinuous in $B'$, see the discussion after Definition \ref{def.PI.}.
Thus Lemma~\ref{lemma:capcomparison} shows that there is an open set $G\subset B'$
such that $\cp(G,B)<\eps$ and $v|_{B'\setm G}$ is continuous.
As $\eps<\cp(\{x_0\},B)$, we see that $x_0\notin G$ and $v|_{B'\setm G}$
is continuous at $x_0$.
Thus, there exists $\rho>0$ such that $B(x_0,\rho)\subset  B'$
and $v<1$ in $B(x_0,\rho)\setm G$.
Since $v=1$ in $E\cap B'$, we must have $E\cap B(x_0,\rho)\subset G$,
and hence
\begin{equation*}  
\cp(E\cap B(x_0,\rho),B) \le \cp(G,B) <\eps.\qedhere
\end{equation*}
\end{proof}

As a corollary of Lemma~\ref{lem-thin-Cp-to-0} we obtain the following strong
Cartan property at points of positive capacity, which also gives a new
characterization of thin sets at such points.

\begin{prop}  \label{prop-char-thin-cap-pos}
Assume that $\Cp(\{x_0\})>0$ and that $x_0\in\itoverline{E}\setm E$.
Then the following are equivalent.
\begin{enumerate}
\item \label{Cp-pos-thin}
$E$ is thin at $x_0$\textup{;}
\item \label{Cp-pos-to0}
for every {\rm(}some\/{\rm)} ball $B\ni x_0$ with $\Cp(X\setm B)>0$,
\[
\lim_{\rho\to0} \cp(E\cap B(x_0,\rho),B) =0\textup{;}
\]
\item \label{Cp-pos-superh}
for every {\rm(}some\/{\rm)} ball $B\ni x_0$ with $\Cp(X\setm B)>0$
there exists a nonnegative superharmonic function $u$ in $B$ such that
\[
\lim_{E\ni x\to x_0} u(x) =\infty > u(x_0).
\]
\end{enumerate}
\end{prop}

\begin{remark}
By letting $v:= \min\{u,u(x_0)+1\}$, we obtain a bounded superharmonic
function satisfying~\eqref{eq-Cartan}.
\end{remark}

\begin{proof}
\ref{Cp-pos-thin}  \imp \ref{Cp-pos-to0}
This is a special case of Lemma~\ref{lem-thin-Cp-to-0}.

\ref{Cp-pos-to0}  \imp \ref{Cp-pos-superh}
For $j=1,2,\ldots$, find $r_j>0$ such that
\[
\cp(E\cap B(x_0,r_j),B)<2^{-jp}.
\]
Since $\cp$ is an outer capacity, by Theorem~\ref{thm-outercap-cp},
there exist open sets $G_j\not\ni x_0$
such that $G_j\supset E\cap B(x_0,r_j)$ and $\cp(G_j,B)<2^{-jp}$.
Let $v_j$ be the capacitary potential of $G_j$ in $B$.
The Poincar\'e inequality for $\Np_0$ (also known as Friedrichs' inequality),
see Corollary~5.54 in Bj\"orn--Bj\"orn~\cite{BBbook},
shows that
\[
\int_B v_j^p\,d\mu \le C_B \int_B g_{v_j}^p\,d\mu < C_B 2^{-jp},
\]
and hence $\|v_j\|_{\Np(X)} \le \widetilde{C}_B 2^{-j}$.
It follows that
$v:=\sum_{j=1}^\infty v_j \in\Np_0(B)$.

Let $u$ be the lower semicontinuously
regularized solution of the $\K_{v,0}(B)$-obstacle problem.
Then $u\in\Np_0(B)$ is a nonnegative superharmonic function in $B$ and
(as $G_j$ are open) $u\ge k$ in $G_1\cap \ldots \cap G_k$, $k=1,2,\ldots$\,.
It follows that $\lim_{E\ni x\to x_0} u(x)=\infty$.
On the other hand, as $u\in\Np_0(B)$ and $\Cp(\{x_0\})>0$, we have
$u(x_0)<\infty$ by Definition~\ref{deff-sobcap}.

\ref{Cp-pos-superh}  \imp \ref{Cp-pos-thin}
Since superharmonic functions are finely continuous,
by Theorem~\ref{thm-superharm-p-fine},
the set $U=\{x \in B : u(x) < u(x_0)+1\}$ is finely open.
As $x_0 \in U$, we get that $B \setm U$ is thin at $x_0$,
and hence $E$ is also thin at $x_0$.
\end{proof}

Another consequence of Lemma~\ref{lem-thin-Cp-to-0} is the following
result, which is proved in the same way as the first part of
Lemma~2.138 in Mal\'y--Ziemer~\cite{MZ},
although we use the variational capacity instead of
the Sobolev capacity.
We include a short proof for the reader's convenience.

\begin{lem}   \label{lem-cap-quot-le-eps}
If $E$ is thin at $x_0$ and $\eps>0$, then there exists $\rho>0$ such
that
\[
\int_0^1 \biggl( \frac{\cp(E\cap B(x_0,\rho)\cap B(x_0,r),B(x_0,2r))}
{\cp(B(x_0,r),B(x_0,2r))} \biggr)^{1/(p-1)} \, \frac{dr}{r} < \eps.
\]
\end{lem}

\begin{proof} Lemma~\ref{lem-thin-Cp-to-0} implies that the functions
\[
f_j(r) := \biggl( \frac{\cp(E\cap B(x_0,1/j)\cap B(x_0,r),B(x_0,2r))}
{\cp(B(x_0,r),B(x_0,2r))} \biggr)^{1/(p-1)} \frac1r
\]
decrease pointwise to zero on $(0,1)$.
As $E$ is thin at $x_0$, we see that $f_1$ is integrable on $(0,1)$,
and hence by dominated
convergence, $\int_0^1f_j(r)\,dr\to0$, as $j\to\infty$.
Choosing $\rho=1/j$ for some sufficiently large $j$ concludes the proof.
\end{proof}

Now we can deduce the following result which we will
need when proving Theorem~\ref{thm-char-fine-cont}.

\begin{lem}\label{lem-countable-thin}
Assume that the sets $E_j$, $j=1,2,\ldots$, are thin at $x_0$.
Then there exist radii $r_j>0$ such that the set
\[
E = \bigcup_{j=1}^\infty (E_j\cap B(x_0,r_j))
\]
is thin at $x_0$.
\end{lem}

Note that in general the union $\bigcup_{j=1}^\infty E_j$ need not be thin
at $x_0$. This happens e.g.\ if $E_j=\bdry B(x_0,1/j)$.
To obtain a similar example where $x_0\in \itoverline{E}_j$, $j=1,2,\ldots$,
let $E_j = \bdry B(x_0,1/j) \cup E_0$, where $E_0$ is an arbitrary set thin
at $x_0$ and such that $x_0\in\itoverline{E}_0$.

\begin{proof} 
The proof of the corresponding result for weighted $\R^n$ in
Heinonen--Kilpel\"ainen--Martio~\cite{HeKiMa}, Lemma~12.25,
carries over verbatim to metric spaces.
However, instead of appealing to their Lemma~12.24 (i.e.\
our Lemma~\ref{lem-thin-Cp-to-0}), it is more straightforward
to appeal to our  Lemma~\ref{lem-cap-quot-le-eps}.
\end{proof}

We end this paper with the proof
of Theorem~\ref{thm-char-fine-cont}.

\begin{proof}[Proof of  Theorem~\ref{thm-char-fine-cont}]
\ref{it-i} $\imp$ \ref{it-iii}
For each $j=1,2,\dots$ there is
a finely open set  $U_j\ni x_0$
such that $|u(x)-u(x_0)|<1/j$ for every $x\in U_j$.
Since the sets $E_j:=X\setminus U_j$ are thin at $x_0$,
Lemma~\ref{lem-countable-thin} implies  that there are radii $r_j>0$
such that the set
\[
E = \bigcup_{j=1}^\infty (E_j\cap B(x_0,r_j))
\]
is thin at $x_0$.
It follows that $|u(x)-u(x_0)|<1/j$ for every
$x\in U\cap B(x_0,r_j)\setminus E$, and we conclude that~\ref{it-iii} holds.

The implication \ref{it-iii} $\imp$ \ref{it-ii} is immediate and
\ref{it-ii} $\imp$ \ref{it-i} follows from Lemma~\ref{lem-fine-intclosure}.
\end{proof}


\begin{thebibliography}{99}

\bibitem{AdHe} \book{\idxauth{Adams}{D. R} \AND \idxauth{Hedberg}{L. I}}
         {Function Spaces and Potential Theory}
         {Springer, Berlin--Heidelberg, 1996}

\bibitem{AdLew} \art{Adams, D. R. \AND Lewis, J. L.}
        {Fine and quasiconnectedness in nonlinear potential theory}
        {Ann. Inst. Fourier\/ \textup{(}Grenoble\/\textup{)}}
        {35{\rm :1}}{1985}{57--73}  

\bibitem{AdMey} \art{Adams, D. R. \AND Meyers, N. G.}
        {Thinness and Wiener criteria for non-linear potentials}
        {Indiana Univ. Math. J.}{22}{1972/73}{169--197}


\bibitem{ABsuper} \art{Bj\"orn, A.}
        {Characterizations of \p-superharmonic
         functions on metric spaces}
        {Studia Math.} {169} {2005} {45--62}

\bibitem{ABkellogg} \art{Bj\"orn, A.}
        {A weak Kellogg property for quasiminimizers}
        {Comment. Math. Helv.}{81}{2006}{809--825}

\bibitem{BB} \art{Bj\"orn, A. \AND Bj\"orn, J.}
        {Boundary regularity for \p-harmonic functions and
          solutions of the obstacle problem}
        {J. Math. Soc. Japan} {58} {2006} {1211--1232}

\bibitem{BBbook} \book{Bj\"orn, A. \AND Bj\"orn, J.}
        {\it Nonlinear Potential Theory on Metric Spaces}
    {EMS Tracts in Mathematics {\bf 17},
        European Math. Soc., Zurich, 2011}

\bibitem{BBnonopen} \arttoappear{\auth{Bj\"orn}{A} \AND \auth{Bj\"orn}{J}}	
	{Obstacle and Dirichlet problems on arbitrary nonopen sets
          in metric spaces, and fine topology}
        {Rev. Mat. Iberoam.}

\bibitem{BBvarcap} \arttoappear{\auth{Bj\"orn}{A} \AND \auth{Bj\"orn}{J}}	
	{The variational capacity with respect to nonopen sets in metric spaces}
	{Potential Anal.\\ {\tt DOI:10.1007/s11118-013-9341-1}}

\bibitem{ringcap} \artprep{\auth{Bj\"orn}{A}, \auth{Bj\"orn}{J}
	\AND \auth{Lehrb\"ack}{J}}
        {Sharp capacity estimates for annuli in weighted $\R^n$ and metric
        spaces}{\emph{In preparation}}

\bibitem{BBP} \art{\auth{Bj\"orn}{A}, \auth{Bj\"orn}{J}
	\AND \auth{Parviainen}{M}}
        {Lebesgue points and the fundamental convergence theorem
        for superharmonic functions on metric spaces}
        {Rev. Mat. Iberoam.} {26} {2010} {147--174}

\bibitem{BBS} \art{Bj\"orn, A., Bj\"orn, J. \AND Shanmugalingam, N.}
        {The Dirichlet problem for \p-harmonic functions on metric spaces}
        {J. Reine Angew. Math.} {556} {2003} {173--203}

\bibitem{BBS2} \art{Bj\"orn, A., Bj\"orn, J. \AND Shan\-mu\-ga\-lin\-gam, N.}
        {The Perron method for \p-harmonic functions}
        {J. Differential Equations} {195} {2003} {398--429}

\bibitem{BBS5} \art{Bj\"orn, A., Bj\"orn, J. \AND Shan\-mu\-ga\-lin\-gam, N.}
        {Quasicontinuity of Newton--Sobolev functions and density of Lipschitz
        functions on metric spaces}
        {Houston J. Math.} {34} {2008} {1197--1211}


\bibitem{BMarola} \art{\idxauth{Bj\"orn}{A} \AND \idxauth{Marola}{N}}
	{Moser iteration\index{Moser's iteration method}
         for (quasi)minimizers on metric spaces}
	{Manuscripta Math.} {121} {2006} {339--366}

\bibitem{Bj}  \art{Bj\"orn, J.}
        {Boundary continuity for quasiminimizers on metric spaces}
        {Illinois J. Math.} {46}  {2002} {383--403}

\bibitem{JB-Matsue} \artin{\idxauth{Bj\"orn}{J}}
        {Wiener criterion for Cheeger \p-harmonic
        functions on metric spaces}
        {{\it Potential Theory in Matsue},
        Advanced Studies in Pure Mathematics {\bf 44}, pp. 103--115,
        Mathematical Society of Japan, Tokyo, 2006}

\bibitem{JB-pfine}  \art{Bj\"orn, J.} {Fine continuity on metric spaces}
        {Manuscripta Math.} {125} {2008} {369--381}

\bibitem{JBCalcVar} \art{\auth{Bj\"orn}{J}}
        {Necessity of a Wiener type condition for boundary regularity
          of quasiminimizers and nonlinear elliptic equations}
        {Calc. Var. Partial Differential Equations}
	{35} {2009} {481--496}


\bibitem{BMS}  \art{Bj\"orn, J., MacManus, P. \AND Shanmugalingam, N.}
        {Fat sets and pointwise boundary estimates for \p-harmonic
        functions in metric spaces}
        {J. Anal. Math.}{85}{2001}{339--369}

\bibitem{FarFennAnn} \art{Farnana, Z.}
        {The double obstacle problem on metric spaces}
        {Ann. Acad. Sci. Fenn. Math.}{34}{2009}{261--277}

\bibitem{Fugl71} \art{\auth{Fuglede}{B}}
   {The quasi topology associated with a countably subadditive set function}
   {Ann. Inst. Fourier\/ \textup{(}Grenoble\/\textup{)}}
   {21{\rm:1}}{1971}{123--169}

\bibitem{Fug} \book{\auth{Fuglede}{B}}
        {Finely Harmonic Functions}
        {Springer, Berlin--New York, 1972}

\bibitem{Haj-PA} \art{\auth{Haj\l asz}{P}}
        {Sobolev spaces on an arbitrary metric space}
        {Potential Anal.}{5}{1996}{403--415}

\bibitem{Haj03} \artin{Haj\l asz, P.}
	{Sobolev spaces on metric-measure spaces}
	{\emph{Heat Kernels and Analysis on Manifolds\textup{,} Graphs and
	Metric Spaces\/ \textup{(}Paris\textup{,} 2002\/\textup{)}},
	Contemp. Math. {\bf 338}, pp. 173--218,
	Amer. Math. Soc., Providence, RI, 2003}

\bibitem{Hedb} \art{Hedberg, L. I.}
   {Non-linear potentials and approximation in the mean by analytic functions}
   {Math. Z.}{129}{1972}{299--319} 

\bibitem{HedWol} \art{Hedberg, L. I. \AND Wolff, T. H.}
        {Thin sets in nonlinear potential theory}
        {Ann. Inst. Fourier\/ \textup{(}Grenoble\/\textup{)}}
        {33{\rm:4}}{1983}{161--187} 

\bibitem{HaKo} \book{\auth{Haj\l asz}{P}
	\AND \auth{Koskela}{P}}
	{Sobolev met Poincar\'e}
	{{Mem. Amer. Math. Soc.} {\bf 145}:688 (2000)}

\bibitem{heinonen} \book{Heinonen, J.}
        {Lectures on Analysis on Metric Spaces}
        {Springer, New York, 2001}

\bibitem{hei-BAMS} \art{\auth{Heinonen}{J}}
	{Nonsmooth calculus}
	{Bull. Amer. Math. Soc.} {44} {2007} {163--232}

\bibitem{HeKiMa89} \art{Heinonen, J., Kilpel\"ainen, T.\ \AND Martio, O.}
        {Fine topology and quasilinear elliptic equations}
        {Ann. Inst. Fourier\/ \textup{(}Grenoble\/\textup{)}}
        {39{\rm :2}} {1989} {293--318}

\bibitem{HeKiMa} \book{\auth{Heinonen}{J},
	\auth{Kilpel\"ainen}{T}
	\AND \auth{Martio}{O}}
        {Nonlinear Potential Theory of Degenerate Elliptic Equations}
        {2nd ed., Dover, Mineola, NY, 2006}

\bibitem{HeKo98} \art{Heinonen, J. \AND Koskela, P.}
	{Quasiconformal maps in metric spaces with controlled geometry}
	{Acta Math.} {181} {1998} {1--61}

\bibitem{HKST} \book{\idxauth{Heinonen}{J}, \idxauth{Koskela}{P},
	\idxauth{Shanmugalingam}{N} \AND \idxauth{Tyson}{J. T}}
       {Sobolev Spaces on Metric Measure Spaces\/\textup{:}
          an Approach Based on Upper Gradients}
	{In preparation}


\bibitem{Ki89} \art{Kilpel\"ainen, T.}
        {Potential theory for supersolutions of degenerate elliptic equations}
        {Indiana Univ. Math. J.}
        {38} {1989} {253--275}

\bibitem{KiMa92} \art{Kilpel\"ainen, T. \AND Mal\'y, J.}
        {Supersolutions to degenerate elliptic equation on quasi open sets}
        {Comm. Partial Differential Equations}
        {17} {1992} {371--405}

\bibitem{KiMa} \art{Kilpel\"ainen, T. \AND Mal\'y, J.}
{The Wiener test and potential estimates for quasilinear elliptic equations}
{Acta Math.}{172} {1994}{137--161}


\bibitem{KiLa} \artin{\idxauth{Kinnunen}{J} \AND \idxauth{Latvala}{V}}
        {Fine regularity of superharmonic functions on metric
          spaces}
        {{\it Future Trends in Geometric Function Theory
           RNC Workshop Jyv\"askyl\"a 2003},
           Rep. Univ. Jyv\"askyl\"a Dep. Math. Stat. {\bf  92},
           pp. 157--167,  University of  Jyv\"askyl\"a, Jyv\"askyl\"a, 2003}


\bibitem{KiMa02} \art{Kinnunen, J. \AND Martio, O.}
        {Nonlinear potential theory on metric spaces}
        {Illinois Math. J.} {46} {2002} {857--883}

\bibitem{KiMasuperh} \art{Kinnunen, J. \AND Martio, O.}
        {Sobolev space properties of superharmonic functions on metric spaces}
        {Results Math.} {44}{2003}{114--129}

\bibitem{KiSh01} \art{Kinnunen, J. \AND Shanmugalingam, N.}
        {Regularity of quasi-minimizers on metric spaces}
        {Manuscripta Math.} {105} {2001} {401--423}

\bibitem{korte08} \art{Korte, R.}
        {A Caccioppoli estimate and fine continuity for superminimizers
         on metric spaces}
        {Ann. Acad. Sci. Fenn. Math.} {33} {2008} {597--604}

\bibitem{KoMc} \art{Koskela, P. \AND MacManus, P.}
        {Quasiconformal mappings and Sobolev spaces}
        {Studia Math.}{131}{1998}{1--17}

\bibitem{LatPhD} \book{\auth{Latvala}{V}}
        {Finely Superharmonic Functions of Degenerate Elliptic Equations}
        {Ann. Acad. Sci. Fenn. Ser. A I Math. Dissertationes {\bf 96}
        {(1994)}}

\bibitem{Lat97} \art{Latvala, V.}
        {Fine topology and $A_p$-harmonic morphisms}
        {Ann. Acad. Sci. Fenn. Math.} {22} {1997} {237--244}

\bibitem{Lat00} \art{Latvala, V.}
        {A theorem on fine connectedness}
        {Potential Anal.} {12} {2000} {221--232}

\bibitem{Lind-Mar} \art{\idxauth{Lindqvist}{P} \AND
        \idxauth{Martio}{O}} 
        {Two theorems of N. Wiener for solutions of quasilinear
          elliptic equations} 
        {Acta Math.}{155}{1985}{153--171}

\bibitem{LuMaZa} \book{\auth{Luke\v{s}}{J}, \auth{Mal\'y}{J} \AND
        \auth{Zaj\'i\v{c}ek}{L}}
        {Fine Topology Methods in Real Analysis and Potential Theory}
        {Springer, Berlin--New York, 1986}


\bibitem{MZ} \book{Mal\'y, J. \AND Ziemer, W. P.}
{Fine Regularity of Solutions of Elliptic Partial Differential Equations}
{Amer. Math. Soc., Providence, RI, 1997}

\bibitem{Maz70} \artnopt{\auth{Maz{\cprime}ya}{V. G}}
        {On the continuity at a boundary point of solutions of quasi-linear
        elliptic equations}
        {Vestnik Leningrad. Univ. Mat. Mekh. Astronom.}
        {25{\rm:13}} {1970} {42--55}  (Russian).
        English transl.: {\it Vestnik Leningrad Univ. Math.}
        {\bf 3} (1976), 225--242.

\bibitem{MazHa70} \artnopt{Maz\cprime ya, V. G. \AND Havin, V. P.}
        {A nonlinear analogue of the Newtonian potential, and metric 
        properties of $(p,l)$-capacity}
        {Dokl. Akad. Nauk SSSR}{194}{1970}{770--773} (Russian).
        English transl.: {\it Soviet Math. Dokl.} {\bf 11} (1970), 1294--1298.

\bibitem{MazHa72} \artnopt{Maz\cprime ya, V. G. \AND Havin, V. P.}
        {A nonlinear potential theory}
        {Uspekhi Mat. Nauk}{27{\rm:6}}{1972}{67--138} (Russian).
        English transl.: {\it Russian Math. Surveys} {\bf 27{\rm:6}} (1974), 71--148.

\bibitem{Mey75} \art{Meyers, N. G.}
        {Continuity properties of potentials}
        {Duke Math. J}{42}{1975}{157--166}

\bibitem{Mikkonen} \book{Mikkonen, P.}
    {On the Wolff Potential and Quasilinear Elliptic Equations Involving
    Measures}
    {Ann. Acad. Sci. Fenn. Math. Diss. {\bf 104} (1996)}

\bibitem{Sh-rev} \art{Shanmugalingam, N.}
        {Newtonian spaces\textup{:} An extension of Sobolev spaces
        to metric measure spaces}
        {Rev. Mat. Iberoam.}{16}{2000}{243--279}

\bibitem{Sh-harm} \art{Shanmugalingam, N.}
        {Harmonic functions on metric spaces}
        {Illinois J. Math.}{45}{2001}{1021--1050}

\end{thebibliography}
\end{document}